\documentclass[final,3p,times]{elsarticle}

\usepackage{amsmath}
\usepackage{amsfonts}
\usepackage{amssymb}
\usepackage[english]{babel}
\usepackage{color}
\usepackage{graphicx}
\usepackage{lmodern}
\usepackage{amsthm}
\usepackage{mathrsfs}
\usepackage{microtype}
\usepackage{mathscinet}
\usepackage{enumitem}
\usepackage[cal=boondoxo,bb=ams]{mathalfa}
\usepackage{hyperref}
\hypersetup{hidelinks}

\newtheorem{thm}{Theorem}[section]
 
 \newtheorem{lem}[thm]{Lemma}
 \newtheorem{prop}[thm]{Proposition}
 \newtheorem{defn}[thm]{Definition}
 \newtheorem{rem}[thm]{Remark}

\DeclareMathOperator\supp{supp}

\DeclareMathOperator\meas{meas}
\DeclareMathOperator\Fuj{Fuj}

\DeclareMathOperator\loc{loc}

\begin{document}

\begin{frontmatter}



\title{A note on a conjecture for the critical curve of a weakly coupled system of semilinear wave equations with scale-invariant lower order terms}


\author[Frei]{Alessandro Palmieri}
\ead{alessandro.palmieri.math@gmail.com}

\address[Frei]{Institute of Applied Analysis, Faculty for Mathematics and Computer Science, Technical University Bergakademie Freiberg, Pr\"{u}ferstra{\ss}e 9, 09596, Freiberg, Germany}

\begin{abstract}

In this note two blow-up results are proved for a weakly coupled system of semilinear wave equations with distinct scale-invariant lower order terms both in the subcritical case and in the critical case, when the damping and the mass terms make both equations in some sense ``wave-like''. In the proof of the subcritical case an iteration argument is used. This approach is based on a coupled system of nonlinear ordinary integral inequalities and lower bound estimates for the spatial integral of the nonlinearities. In the critical case we employ a test function type method, that has been developed recently by Ikeda-Sobajima-Wakasa and relies strongly on a family of certain self-similar solutions of the adjoint linear equation. Therefore, as critical curve in the $p$ - $q$ plane of the exponents of the power nonlinearities for this weakly coupled system we conjecture a shift of the critical curve for the corresponding weakly coupled system of semilinear wave equations.

\end{abstract}


\end{frontmatter}


\section{Introduction}

In this paper we consider the weakly coupled system of wave equations with scale-invariant damping and mass terms with different multiplicative constants in the lower order terms and with power nonlinearities, namely,
\begin{align}\label{weakly coupled system}
\begin{cases}
u_{tt}-\Delta u +\frac{\mu_1}{1+t}u_t +\frac{\nu_1^2}{(1+t)^2}u = |v|^p,  & x\in \mathbb{R}^n, \ t>0,  \\
v_{tt}-\Delta v +\frac{\mu_2}{1+t}v_t +\frac{\nu_2^2}{(1+t)^2}v = |u|^q,  & x\in \mathbb{R}^n, \ t>0, \\
 (u,u_t,v,v_t)(0,x)= (\varepsilon u_0, \varepsilon u_1, \varepsilon v_0, \varepsilon v_1)(x) & x\in \mathbb{R}^n,
\end{cases}
\end{align}
where $\mu_1,\mu_2,\nu_1^2,\nu_2^2$ are nonnegative constants, $\varepsilon$ is a positive parameter describing the size of initial data and $p,q>1$. 

Recently, the Cauchy problem for a semilinear wave equation with scale-invariant damping and mass
\begin{align}\label{scale inv eq}
\begin{cases}
u_{tt}-\Delta u +\frac{\mu}{1+t}u_t +\frac{\nu^2}{(1+t)^2}u = |u|^p,  & x\in \mathbb{R}^n, \ t>0,  \\
(u,u_t)(0,x)= (u_0,u_1)(x),& x\in \mathbb{R}^n, 
\end{cases}
\end{align} where $\mu,\nu^2$ are nonnegative constants and $p>1$, has attracted a lot of attention. 
The value of $\delta \doteq (\mu -1) ^2-4\nu^2$ has a strong influence on some properties of solutions to \eqref{scale inv eq} and to the corresponding homogeneous linear equation. According to \cite{Abb15,Wakasugi14,DLR15,DabbLuc15,Wakasa16,NPR16,PalRei17,Pal17,LTW17,IS17,PR17vs,TL1709,TL1711,Pal18odd,Pal18even,DabbPal18,PT18,KatoSak18,Lai18} for $\delta\geqslant 0$ the model in \eqref{scale inv eq} is somehow an intermediate model between the semilinear free wave equation and the semilinear classical damped equation, whose critical exponent is $p_{\text{Fuj}}(n+\frac{\mu-1}{2}-\frac{\sqrt{\delta}}{2})$ for $\delta\geq (n+1)^2$ and seems reasonably to be $p_0(n+\mu)$ for small and nonnegative values of delta, where $p_{\Fuj}(n)$ and $p_0(n)$ denote the Fujita exponent and the Strauss exponent, respectively.

As for the single semilinear wave equation with scale-invariant damping and mass term, the quantities
\begin{align}\label{def deltas}
\delta_j \doteq (\mu_j-1)^2 -4\nu_j^2, \qquad j=1,2,
\end{align}  
play a fundamental role in the description of some of the properties of the solutions to \eqref{weakly coupled system} as, for example, the critical curve. In particular, in \cite{ChenPal18} the critical curve for \eqref{weakly coupled system} is proved to be 
\begin{align}\label{critical exponent parablic like case system}
\max\bigg\{\frac{p+1}{pq-1}-\frac{1}{2}\bigg(\frac{\mu_1-1}{2}-\frac{\sqrt{\delta_1}}{2}\bigg),\frac{q+1}{pq-1}-\frac{1}{2}\bigg(\frac{\mu_2-1}{2}-\frac{\sqrt{\delta_2}}{2}\bigg)\bigg\}=\frac{n}{2}
\end{align} in the case $\delta_1,\delta_2\geqslant (n+1)^2$. Let us remark that \eqref{critical exponent parablic like case system} is a shift of the critical curve in the $p$ - $q$ plane for the weakly coupled system of semilinear classical damped equation with power nonlinearities, which is (cf. \cite{SunWang2007,Narazaki2009,Nishi12,NishiharaWakasugi,NishiWak15}) $$\frac{\max\{p,q\}+1}{pq-1}=\frac{n}{2}.$$

This paper is devoted to the proof of a blow-up results for \eqref{weakly coupled system} in the case $\delta_1,\delta_2\geqslant 0$ both in the subcritical case and on the critical curve. Analogously to what happens in the case of single equations, when $\delta_1,\delta_2$ are small the model is somehow ``wave-like''. Therefore, the blow-up result that we will prove may be optimal only for small values of $\delta_1,\delta_2$ according to the above mentioned papers, where \eqref{scale inv eq} is considered. This is reasonable since we obtain as ``critical curve''
\begin{align*}
\max\left\{\frac{p+2+q^{-1}}{pq-1}-\frac{\mu_1}{2},\frac{q+2+p^{-1}}{pq-1}-\frac{\mu_2}{2}\right\}=\frac{n-1}{2}
\end{align*} which is a generally asymmetric shift of the critical curve for the weakly coupled system of semilinear wave equation with power nonlinearities (see also \cite{DelS97,DGM,DM,AKT00,KT03,Kur05,GTZ06,KTW12}), namely,
\begin{align} \label{crit curv weakly coupled system wave}
\max\left\{\frac{p+2+q^{-1}}{pq-1},\frac{q+2+p^{-1}}{pq-1}\right\}=\frac{n-1}{2}.
\end{align}

Before stating the main results of this paper, let us introduce a suitable notion of energy solutions according to \cite{LTW17}.

\begin{defn} \label{def energ sol intro} Let $u_0,v_0\in H^1(\mathbb{R}^n)$ and $u_1,v_1\in L^2(\mathbb{R}^n)$.
We say that $(u,v)$ is an energy solution of \eqref{weakly coupled system} on $[0,T)$ if
\begin{align*}
& u\in \mathcal{C}([0,T),H^1(\mathbb{R}^n))\cap \mathcal{C}^1([0,T),L^2(\mathbb{R}^n))\cap L^q_{loc}(\mathbb{R}^n\times[0,T)), \\
& v\in \mathcal{C}([0,T),H^1(\mathbb{R}^n))\cap \mathcal{C}^1([0,T),L^2(\mathbb{R}^n))\cap L^p_{loc}(\mathbb{R}^n\times[0,T))
\end{align*}
satisfy $u(0,x)=\varepsilon u_0(x)$ and $v(0,x)=\varepsilon v_0(x)$ in $H^1(\mathbb{R}^n)$ and the equalities
\begin{align} 
&\int_{\mathbb{R}^n}u_t(t,x)\phi(t,x)\,dx-\int_{\mathbb{R}^n}u_t(0,x)\phi(0,x)\,dx - \int_0^t\int_{\mathbb{R}^n} u_t(s,x)\phi_t(s,x) \,dx\, ds\notag \\
& \ +\int_0^t\int_{\mathbb{R}^n}\nabla u(s,x)\cdot\nabla\phi(s,x)\, dx\, ds+\int_0^t\int_{\mathbb{R}^n}\left(\frac{\mu_1 }{1+s} u_t(s,x)+\frac{\nu_1^2}{(1+s)^2}u(s,x)\right)\phi(s,x)\,dx\, ds  \notag \\
& =\int_0^t \int_{\mathbb{R}^n}|v(s,x)|^p\phi(s,x)\,dx \, ds  \label{def u}
\end{align} and 
\begin{align} 
&\int_{\mathbb{R}^n}v_t(t,x)\psi(t,x)\,dx-\int_{\mathbb{R}^n}v_t(0,x)\psi(0,x)\,dx - \int_0^t\int_{\mathbb{R}^n} v_t(s,x)\psi_t(s,x) \,dx\, ds\notag \\
& \ +\int_0^t\int_{\mathbb{R}^n}\nabla v(s,x)\cdot\nabla\psi(s,x)\, dx\, ds+\int_0^t\int_{\mathbb{R}^n}\left(\frac{\mu_2 }{1+s} v_t(s,x)+\frac{\nu_2^2}{(1+s)^2}v(s,x)\right)\psi(s,x)\,dx\, ds  \notag \\
& =\int_0^t \int_{\mathbb{R}^n}|u(s,x)|^q\psi(s,x)\,dx \, ds  \label{def v}
\end{align}
for any $\phi,\psi \in \mathcal{C}_0^\infty([0,T)\times\mathbb{R}^n)$ and any $t\in [0,T)$.
\end{defn}
 After a further integration by parts in \eqref{def u} and \eqref{def v}, letting $t\rightarrow T$, we find that $(u,v)$ fulfills the definition of weak solution to \eqref{weakly coupled system}.
 
We can now state the main theorem in the subcritical case.

\begin{thm}\label{Thm blowup iteration} Let $\mu_1,\mu_2,\nu_1^2,\nu_2^2$ be nonnegative constants such that $\delta_1, \delta_2 \geqslant 0$. Let us consider $p,q>1$ satisfying 
\begin{align}\label{critical exponent wave like case system}
\max\left\{\frac{p+2+q^{-1}}{pq-1}-\frac{\mu_1}{2},\frac{q+2+p^{-1}}{pq-1}-\frac{\mu_2}{2}\right\}>\frac{n-1}{2}\,.
\end{align} 

Assume that $u_0,v_0\in H^1(\mathbb{R}^n)$ and $ u_1,v_1\in  L^2(\mathbb{R}^n)$ are compactly supported in $B_R\doteq \{x\in \mathbb{R}^n: |x|\leqslant R\}$ and satisfy
\begin{align}
u_0(x)\geqslant 0  \ \ \mbox{and} \ \ u_1(x)+\tfrac{\mu_1-1-\sqrt{\delta_1}}{2}u_0(x)\geqslant 0 , \label{assumptions u0, u1} \\
v_0(x)\geqslant 0  \ \ \mbox{and} \ \ v_1(x)+\tfrac{\mu_2-1-\sqrt{\delta_2}}{2}v_0(x)\geqslant 0 . \label{assumptions v0, v1}
\end{align}

Let $(u,v)$ be an energy solution of \eqref{weakly coupled system} with lifespan $T=T(\varepsilon)$ according to Definition \ref{def energ sol intro}. Then, there exists a positive constant $\varepsilon_0=\varepsilon_0(u_0,u_1,v_0,v_1,n,p,q,\mu_1,\mu_2,\nu_1^2,\nu_2^2,R)$ such that for any $\varepsilon\in (0,\varepsilon_0]$ the solution $(u,v)$ blows up in finite time. Moreover,
 the upper bound estimate for the lifespan
\begin{align} \label{lifespan upper bound estimate}
T(\varepsilon)\leqslant C \varepsilon ^{-\max\{F(n+\mu_1,p,q),F(n+\mu_2,q,p)\}^{-1}}
\end{align} holds, where C is an independent of $\varepsilon$, positive constant and 
\begin{align}\label{def F(n,p,q) function}
F(n,p,q)\doteq \frac{p+2+q^{-1}}{pq-1}-\frac{n-1}{2}.
\end{align}
\end{thm}

\begin{rem} For $\mu_1=\mu_2=0$ the previous upper bound for the lifespan coincides with the sharp estimate for the lifespan of local solutions to the weakly coupled system of semilinear wave equations with power nonlinearities in the subcritical case. However, as we do not deal with global in time existence results for \eqref{weakly coupled system} in the present work, we do not derive a lower bound estimate for $T(\varepsilon)$. Let us underline that the shift in the first argument of $F=F(n,p,q)$ corresponds to the shift in the critical curve.
\end{rem}

Let us state the main result in the critical case.

\begin{thm}\label{Thm critical case}
Let $\mu_1,\mu_2,\nu_1^2,\nu_2^2$ be nonnegative constants such that $\delta_1, \delta_2 \geqslant 0$. Let us consider $p,q>1$ satisfying 
\begin{align}\label{critical exponent wave like case system, critical case}
\max\left\{\frac{p+2+q^{-1}}{pq-1}-\frac{\mu_1}{2},\frac{q+2+p^{-1}}{pq-1}-\frac{\mu_2}{2}\right\}=\frac{n-1}{2}\,
\end{align} and 
\begin{align}\label{technical restrictions on p,q critical case}
\tfrac{1}{p}<\tfrac{n-\sqrt{\delta_2}}{2} \, , \ \ \tfrac{1}{q}<\tfrac{n-\sqrt{\delta_1}}{2}.
\end{align}

Assume that $u_0,v_0\in H^1(\mathbb{R}^n)$ and $u_1,v_1\in L^2(\mathbb{R}^n)$ are nonnegative, pairwise nontrivial and compactly supported in $B_{r_0}$, with $r_0\in (0,1)$ .

Let $(u,v)$ be an energy solution of \eqref{weakly coupled system} with lifespan $T=T(\varepsilon)$. Then, there exists a positive constant $\varepsilon_0=\varepsilon_0(u_0,u_1,v_0,v_1,n,p,q,\mu_1,\mu_2,\nu_1^2,\nu_2^2,r_0)$ such that for any $\varepsilon\in (0,\varepsilon_0]$ the solution $(u,v)$ blows up in finite time. Moreover,
 the upper bound estimates for the lifespan
\begin{align} \label{lifespan upper bound estimate, critical case}
T(\varepsilon)\leqslant 
\begin{cases} 
\exp \big(C \varepsilon ^{-q(pq-1)}\big) & \mbox{if} \ \ 0=F(n+\mu_1,p,q)>F(n+\mu_2,q,p), \\
\exp \big(C \varepsilon ^{-p(pq-1)}\big) & \mbox{if} \ \ 0=F(n+\mu_2,q,p)>F(n+\mu_1,p,q), \\
\exp \big(C \varepsilon ^{-(pq-1)}\big) & \mbox{if} \ \ 0=F(n+\mu_1,p,q)=F(n+\mu_2,q,p), \\
\exp \big(C \varepsilon ^{-p(p-1)}\big) & \mbox{if} \ \ p=q=p_0(n+\mu_1) \ \ \mbox{and} \ \ \mu_1=\mu_2, \, \nu_1^2=\nu_2^2
\end{cases}
\end{align} hold, where C is an independent of $\varepsilon$, positive constant and $F=F(n,p,q)$ is defined by \eqref{def F(n,p,q) function}.
\end{thm}

\begin{rem} In \eqref{lifespan upper bound estimate, critical case} the last case corresponds to the case in which $ 0=F(n+\mu_1,p,q)=F(n+\mu_2,q,p)$ and the scale-invariant terms  in \eqref{weakly coupled system} have the same coefficients (same partial differential operator on the left hand sides).
\end{rem}

\begin{rem} As we will see in the proof of Theorem \ref{Thm critical case}, the conditions \eqref{technical restrictions on p,q critical case} are technical requirements, which guarantee the nonemptiness of the ranges for certain parameters. Nonetheless, in dimension $n\geqslant 3$ and for $0\leqslant \delta_1,\delta_2\leqslant (n-2)^2$ the assumption on the exponents $p,q$ given by \eqref{technical restrictions on p,q critical case} is trivially satisfied for any $p,q>1$.
\end{rem}

The remaining part of this paper is organized as follows: in Section \ref{Section lower bounds} we present a solution to the corresponding adjoint linear homogeneous system, whose components have separated variables,  and we derive some lower bounds for certain functionals related to a local solution; then, in Section \ref{Section proof main thm} we prove Theorem \ref{Thm blowup iteration} using the preliminary results proved in Section \ref{Section lower bounds}. In Section \ref{Section supersol} we introduce the notion of super-solutions of the wave equation with scale-invariant damping and mass and we derive some estimates for them. A family of self-similar solutions of the adjoint equation of the linear wave equation with scale-invariant damping and mass and their properties are shown in Section \ref{Section self similar sol}. Finally, Theorem \ref{Thm critical case} is proved in Section \ref{Section critical case}. Let us underline explicitly that besides the notations that have been introduced in this introduction, the notations in Sections \ref{Section lower bounds}-\ref{Section proof main thm} (subcritical case) and the notations in Sections \ref{Section supersol}-\ref{Section self similar sol}-\ref{Section critical case} (critical case) are mutually independent and they should be not compared or overlapped by the reader.

\subsection*{Notations} Throughout this paper we will use the following notations: $B_R$ denotes the ball around the origin with radius $R$; $f \lesssim g$ means that there exists a positive constant $C$ such that $f \leqslant Cg$ and, similarly, for $f\gtrsim g$; moreover, $f\approx g$ means $f\lesssim g$ and $f\gtrsim g$; finally, as in the introduction, $p_0(n)$ denotes the Strauss exponent.

\section{Solution of the adjoint linear problem and preliminaries}
\label{Section lower bounds}

The arguments used in this section are the generalization for a weakly coupled system of those used in \cite[Section 2]{PT18} for a single equation. 
 
Before starting with the construction of a solution to the adjoint system to homogeneous system of scale-invariant wave equations, that is, a solution of the system
\begin{align} \label{adjoint hom system}
\begin{cases}
\Phi_{tt}-\Delta \Phi - \partial_t\Big(\frac{\mu_1}{1+t}\Phi\Big) +\frac{\nu_1^2}{(1+t)^2}\Phi = 0,  & x\in \mathbb{R}^n, \ t>0,  \\
\Psi_{tt}-\Delta \Psi - \partial_t\Big(\frac{\mu_2}{1+t}\Psi\Big) +\frac{\nu_2^2}{(1+t)^2}\Psi = 0,  & x\in \mathbb{R}^n, \ t>0,
\end{cases}
\end{align} we recall the definition of the modified Bessel function of the second kind of order $\varsigma$
$$K_{\varsigma}(t)=\int_0^{\infty}\exp(-t\cosh z)\cosh(\varsigma z)dz,\ \ \varsigma\in \mathbb{R}$$
which is a solution of the equation
$$\bigg(t^2\frac{d^2}{dt^2}+t\frac{d}{dt}-(t^2+\varsigma^2)\bigg)K_{\varsigma}(t)=0, \ \ t>0.$$
We collect some important properties concerning $K_\varsigma(t)$ in the case in which $\varsigma$ is a real parameter. Interested reader may refer to \cite{Erdelyi}.
 On the one hand, the following asymptotic behavior of $K_{\varsigma}(t)$ holds:
\begin{align}\label{K1}
K_{\varsigma}(t) & = \sqrt{\dfrac{\pi}{2t}}\, e^{-t}[1+O(t^{-1})]  \ \ \ \mbox{as}\ \ \  t\to \infty.
\end{align}

On the other hand, the following derivative identity holds:
\begin{align}
\frac{d}{dt}K_{\varsigma}(t) & = -K_{\varsigma+1}(t)+\frac{\varsigma}{t}K_{\varsigma}(t). \label{K4} 
\end{align}

As we will construct a solution $(\Phi,\Psi)$ with separated variables, firstly, we set the auxiliary functions with respect to the time variable, namely,
\begin{align*}
\lambda_j(t)& \doteq (1+t)^{\frac{\mu_j+1}{2}}K_{\varsigma_j}(1+t) \ \ \ \mbox{for} \ \  t\geqslant 0  \ \ \mbox{and} \ \ j=1,2,
\end{align*} where $\varsigma_j= \frac{\sqrt{\delta_j}}{2}$ for $j=1,2$.
It is clear by direct computations that $\lambda_1,\lambda_2$ satisfy
\begin{equation}
\label{lamb}
\bigg(\frac{d^2}{dt^2}-\frac{\mu_j}{1+t}\frac{d}{dt}+\frac{\mu_j+\nu_j^2}{(1+t)^2}-1\bigg)\lambda_j(t)=0 \ \  \mbox{for} \ \ t>0 \ \ \mbox{and} \ \ j=1,2.\\
\end{equation}

Following \cite{YZ06}, let us introduce the function
\begin{equation*}\varphi(x) \doteq
\begin{cases}
\int_{\mathbb{S}^{n-1}}e^{x\cdot \omega}d\omega \ \ &\mbox{if}\ \ n\geqslant2,\\
e^x+e^{-x} \ \ &\mbox{if}\ \ n=1.
\end{cases}
\end{equation*}
 The function $\varphi$  satisfies $$\Delta\varphi(x)=\varphi(x) \ \ \ \mbox{for} \ \  x\in \mathbb{R}^n $$ and the asymptotic estimate
\begin{equation}
\varphi(x)\sim C_n|x|^{-\frac{n-1}2}e^{|x|} \ \ \ \mbox{as}\ \ \ |x|\rightarrow\infty.
\end{equation}
We may introduce now the functions $\Phi,\Psi$
\begin{align*}
\Phi(t,x) &  \doteq \lambda_1(t)\, \varphi(x), \\
\Psi(t,x) &  \doteq \lambda_2(t)\, \varphi(x),
\end{align*} which constitute a solution to the adjoint system  \eqref{adjoint hom system}.

The remaining part of the section is devoted to determine lower bounds for $\int_{\mathbb{R}^n}|v(x,t)|^pdx$ and $\int_{\mathbb{R}^n}|u(x,t)|^qdx$.
\begin{lem}
Let us assume that $u_0, u_1,  v_0,  v_1$ are compactly supported in $B_R$ for some $R>0$ and that \eqref{assumptions u0, u1}, \eqref{assumptions v0, v1} are fulfilled. Then, a local energy solution $(u,v)$ satisfies $$\supp u, \, \supp v \subset\{(t,x)\in  [0,T)\times\mathbb{R}^n: |x|\leqslant t+R \}  $$ and
there exists a large $T_0$, which is independent of  $u_0, u_1,  v_0,  v_1$ and $\varepsilon$, such that for any $t>T_0$ and $p,q>1$, the following estimates hold:
\begin{align}\label{Priori u^q}
\int_{\mathbb{R}^n}|u(t,x)|^q dx & \geqslant C_1\varepsilon^q(1+t)^{n-1-\frac{n+\mu_1-1}2 q}, \\
\label{Priori v^p}
\int_{\mathbb{R}^n}|v(t,x)|^p dx & \geqslant K_1\varepsilon^p(1+t)^{n-1-\frac{n+\mu_2-1}2 p},
\end{align}
where $C_1=C_1(u_0,u_1,\varphi,q,R)>0$ and $K_1=K_1(v_0,v_1,\varphi,p,R)>0$ are independent of $\varepsilon$ and $t$.
\end{lem}

\begin{proof} We begin with \eqref{Priori u^q}.
Let us define the functional
\begin{align*}
F(t)\doteq \int_{\mathbb{R}^n}u(t,x)\Phi(t,x)dx 
\end{align*}
with $\Phi$ defined as above. Then, by H\"{o}lder inequality, we have
\begin{equation}\int_{\mathbb{R}^n}|u(t,x)|^qdx\geqslant |F(t)|^q\bigg(\int_{|x|\leqslant t+R}\Phi^{q'}(t,x)dx\bigg)^{-(q-1)},\label{factor}
\end{equation} where $q'$ denotes the conjugate exponent of $q$.

The next step is to determine a lower bound for $F(t)$ and an upper bound for  $\int_{|x|\leqslant t+R}\Phi^{q'}(t,x)dx$,  respectively.
Due to the support property for $u$, we can apply the definition of weak solution with test function $\Phi$. So, for any $t\in (0,T)$ we have
\begin{align*}
 &  \int_{\mathbb{R}^n}\Big(u_t(s,x)\Phi(s,x)-u(s,x)\Phi_t(s,x)+\frac{\mu_1}{1+s}u(s,x)\Phi(s,x)\Big) dx \ \bigg|_{s=0}^{s=t}  \\ & \quad + \int_0^t\int_{\mathbb{R}^n} u(s,x) \Big(\Phi_{ss}(s,x)-\Delta \Phi(s,x)-\partial_s\Big(\frac{\mu_1}{1+s}\Phi(s,x)\Big)+\frac{\nu_1^2}{(1+s)^2}\Phi(s,x)\Big) dx ds 	\\ & =\int_0^t\int_{\mathbb{R}^n}|v(s,x)|^p\Phi(s,x) \,dxds . 
\end{align*}

As the product $|v|^p \, \Phi$ is nonnegative and $\Phi$ solves the first equation in \eqref{adjoint hom system}, from the previous equality we obtain
\begin{align*}
F'(t)+\bigg(\frac{\mu_1}{1+t}-2\frac{\lambda'_1(t)}{\lambda_1(t)}\bigg)F(t)\geqslant \varepsilon \int_{\mathbb{R}^n}\Big(\lambda_1(0) \, u_1(x)+\big(\mu_1\lambda_1(0)-\lambda'_1(0)\big)\, u_0(x)\Big)\varphi(x)\, dx.
\end{align*}
Using \eqref{K4}, we have
\begin{align*}
\lambda'_1(t) & =\tfrac{\mu_1+1}2(1+t)^{\frac{\mu_1-1}2}K_{\varsigma_1}(1+t)+(1+t)^{\frac{\mu_1+1}2}K'_{\varsigma_1}(1+t)\nonumber\\
& =\tfrac{\mu_1+1}2(1+t)^{\frac{\mu_1-1}2}K_{\varsigma_1}(1+t)+(1+t)^{\frac{\mu_1+1}{2}}\left(-K_{\varsigma_1+1}(1+t)+\tfrac{\varsigma_1}{(1+t)}K_{\varsigma_1}(1+t)\right)\nonumber\\
& =\tfrac{\mu_1+1+\sqrt{\delta_1}}2(1+t)^{\frac{\mu_1-1}2}K_{\varsigma_1}(1+t)-(1+t)^{\frac{\mu_1+1}{2}}K_{\varsigma_1+1}(1+t).
\end{align*}
Also,
\begin{align*}
\lambda'_1(0)&=\tfrac{\mu_1+1+\sqrt{\delta_1}}2K_{\varsigma_1}(1)-K_{\varsigma_1+1}(1), \\
\mu_1\lambda_1(0)-\lambda'_1(0) & =\tfrac{\mu_1-1-\sqrt{\delta_1}}2K_{\varsigma_1}(1)+K_{\varsigma_1+1}(1).
\end{align*}
 Consequently,
\begin{align*}
&\lambda_1(0) \, u_1(x)+\big(\mu_1\lambda_1(0)-\lambda'_1(0)\big)\, u_0(x)=K_{\varsigma_1}(1)\left(u_1(x)+\tfrac{\mu_1-1-\sqrt{\delta_1}}{2}u_0(x)	\right)+K_{\varsigma_1+1}(1)\, u_0(x).
\end{align*}
If we denote $$C(u_0,u_1) \doteq \int_{\mathbb{R}^n}\Big(\lambda(0)\, u_1(x)+\big(\mu_1\lambda(0)-\lambda'(0)\big)\,u_0(x)\Big)\varphi(x)\, dx,$$
then, since we assume that $u_0$ and $u_1$ are compactly supported and satisfy \eqref{assumptions u0, u1}, $C(u_0,u_1)$ is finite and positive.
Therefore, we conclude that $F$ satisfies the differential inequality
$$F'(t)+\bigg(\frac{\mu_1}{1+t}-2\frac{\lambda'_1(t)}{\lambda_1(t)}\bigg)F(t)\geqslant\varepsilon \, C(u_0,u_1).$$
Multiplying by $\frac{(1+t)^{\mu_1}}{\lambda_1^2(t)}$  both sides of the previous inequality and then integrating over $[0,t]$, we derive
$$F(t)\geqslant\varepsilon \, C(u_0,u_1)\,\frac{\lambda_1^2(t)}{(1+t)^{\mu_1}}\int_0^t\frac{(1+s)^{\mu_1}}{\lambda_1^2(s)}\, ds.$$
Inserting $\lambda_1(t)=(1+t)^{\frac{\mu_1+1}{2}}K_{\varsigma_1}(1+t)$, we obtain as  lower bound for $F$
\begin{equation} F(t)\geqslant\varepsilon \,  C(u_0,u_1)\int_0^t\frac{(1+t)K^2_{\varsigma_1}(1+t)}{(1+s)K^2_{\varsigma_1}(1+s)}\, ds \,  \geqslant 0.\label{G1}\end{equation}

 The integral involving $\Phi^{q'}$ in the right-hand side of \eqref{factor} can be estimated as in \cite[estimate (2.5)]{YZ06}, namely,
\begin{align}
\int_{|x|\leqslant t+R}\Phi^{q'}(t,x)\, dx & \leqslant \lambda_1^\frac{q}{q-1}(t)\int_{|x|\leqslant t+R}\varphi^{q'}(x)\, dx\notag\\
&\leqslant C_{\varphi,R}(1+t)^{n-1+\big(\frac{\mu_1+1}{2}-\frac{n-1}{2}\big)\frac{q}{q-1}}e^{\frac{q}{q-1}(t+R)}K^{\frac{q}{q-1}}_{\varsigma_1}(1+t),\label{deno}
\end{align}
where $C_{\varphi,R}$ is a suitable positive constant.

Combing the estimate \eqref{G1}, \eqref{deno} and \eqref{factor}, we find
\begin{align*} \int_{\mathbb{R}^n} &|u(x,t)|^qdx \\
&\geqslant C(u_0,u_1)^{q}C_{\varphi,R}^{1-q}\varepsilon^q(1+t)^{q -(n-1)(q-1)-\big(\frac{\mu_1+1}{2}-\frac{n-1}{2}\big)q} e^{-q(t+R)} K^q_{\varsigma_1}(1+t)\bigg(\int_0^t\frac{ds}{(1+s)K^2_{\varsigma_1}(1+s)}\bigg)^q \\
&\geqslant C(u_0,u_1)^{q}C_{\varphi,R}^{1-q}e^{q(1-R)}\varepsilon^q(1+t)^{(2-n-\mu_1)\frac q2+(n-1)} e^{-q(1+t)} K^q_{\varsigma_1}(1+t)\bigg(\int_0^t\frac{ds}{(1+s)K^2_{\varsigma_1}(1+s)}\bigg)^q.
\end{align*}
Due to \eqref{K1}, for a sufficiently large $T_0$ (which is independent of $u_0, u_1,\varepsilon$) and $t>T_0$, we have
$$K^q_{\varsigma_1}(1+t)\sim \bigg(\frac{\pi}{2(1+t)}\bigg)^\frac q2 e^{-q(t+1)}$$
and
\begin{align*}\int_0^t\frac1{(1+s)K^2_{\varsigma_1}(1+s)}ds &\geqslant \frac{2}\pi \int_{t/2}^t  e^{2(1+s)}ds = \frac1\pi\big(e^{2(1+t)}-e^{2+t}\big)\geqslant\frac1{2\pi }e^{2(1+t)}.
\end{align*}
Consequently,
$$\int_{\mathbb{R}^n}|u(t,x)|^qdx\geqslant C_{1}\varepsilon^q(1+t)^{\frac q2(1-n-\mu_1)+(n-1)}\ \ \mbox{for}\ \ t>T_0,$$
where $C_1\doteq 2^{-\frac{3q}{2}}C(u_0,u_1)^{q}C_{\varphi,R}^{1-q}e^{q(1-R)}\pi^{-\frac{q}{2}}.$ The proof of \eqref{Priori v^p} is analogous, as one has to consider the functional $$G(t)\doteq \int_{\mathbb{R}^n}v(t,x)\Psi(t,x)dx$$ instead of $F$ and to use the assumption \eqref{assumptions v0, v1} in place of \eqref{assumptions u0, u1}.  This concludes the proof.
\end{proof}

\section{Subcritical case: Proof of Theorem \ref{Thm blowup iteration}} \label{Section proof main thm}

Let us consider a local solution $(u,v)$ of \eqref{weakly coupled system} on $[0,T)$ and define the couple of time-dependent functionals
\begin{align*}
U(t)   \doteq \int_{\mathbb{R}^n} u(t,x) \, dx ,  \quad
V(t)   \doteq \int_{\mathbb{R}^n} v(t,x) \, dx.
\end{align*}

The proof of Theorem \ref{Thm blowup iteration} is dived in two step. The first step consists in the determination a coupled system of nonlinear ordinary integral inequalities for $U$ and $V$ (iteration frame), while in the second one an iteration argument is used to show the blow-up of $(U,V)$ in finite time. 

\subsubsection*{Determination of the iteration frame}
Let us begin with the first step.

Choosing  $\phi=\phi(s,x)$ and $\psi=\psi(s,x)$ in \eqref{def u} and in \eqref{def v}, respectively, that satisfy $\phi\equiv 1 \equiv \psi$ on $\{(x,s)\in  [0,t]\times \mathbb{R}^n :|x|\leqslant s+R\}$, we obtain
\begin{align*}&\int_{\mathbb{R}^n}u_t(t,x)\,dx-\int_{\mathbb{R}^n}u_t(0,x)\,dx+\int_0^t \int_{\mathbb{R}^n}\bigg(\frac{\mu_1 u_t(s,x)}{1+s}+\frac{\nu^2_1u(s,x)}{(1+s)^2}\bigg)dx \, ds
=\int_0^t \int_{\mathbb{R}^n}|v(s,x)|^pdx \, ds, \\
&\int_{\mathbb{R}^n}v_t(t,x)\,dx-\int_{\mathbb{R}^n}v_t(0,x)\,dx+\int_0^t \int_{\mathbb{R}^n}\bigg(\frac{\mu_2 v_t(s,x)}{1+s}+\frac{\nu^2_2 v(s,x)}{(1+s)^2}\bigg)dx \, ds
=\int_0^t \int_{\mathbb{R}^n}|u(s,x)|^q dx \, ds
\end{align*}
which means that
\begin{align*}
& U'(t)-U'(0)+\int_0^t\frac{\mu_1 }{1+s}U'(s)\,ds+\int_0^t\frac{\nu_1^2}{(1+s)^2}U(s)\,ds=\int_0^t \int_{\mathbb{R}^n}|v(s,x)|^pdx \, ds, \\
& V'(t)-V'(0)+\int_0^t\frac{\mu_2 }{1+s}V'(s)\,ds+\int_0^t\frac{\nu_2^2 }{(1+s)^2}V(s)\,ds=\int_0^t \int_{\mathbb{R}^n}|u(s,x)|^q dx \, ds.
\end{align*}
Differentiating with repect to $t$ the previous equalities, we get
\begin{align}\label{U-Dyn}
U''(t)+\frac{\mu_1}{1+t}U'(t)+\frac{\nu^2_1}{(1+t)^2}U(t)=\int_{\mathbb{R}^n}|v(t,x)|^p dx, \\
\label{V-Dyn}
V''(t)+\frac{\mu_2}{1+t}V'(t)+\frac{\nu^2_2}{(1+t)^2}V(t)=\int_{\mathbb{R}^n}|u(t,x)|^q dx.
\end{align}

Let us consider the quadratic equations
$$r^2-(\mu_1-1)r+\nu_1^2=0, \quad \rho^2-(\mu_2-1)\rho+\nu_2^2=0.$$
Since $\delta_1,\delta_2\geqslant0$ there exit two pair of real roots,
$$r_{1,2}\doteq\frac{\mu_1-1\mp\sqrt{\delta_1}}{2},\ \ \ \rho_{1,2}\doteq\frac{\mu_2-1\mp\sqrt{\delta_2}}{2}.$$
Clearly, if $\mu_1>1$ and $\mu_2>1$, then, $r_{1,2}$ and $\rho_{1,2}$ are positive. Else, if $0\leqslant\mu_1<1$ or $0\leqslant\mu_2<1$, then, $r_{1,2}$ or $\rho_{1,2}$ are negative. When $\mu_1=1$, then, $\nu_1=0$ as $\delta_1\geqslant0$ and, hence, $r_1=r_2=0$. Similarly, if $\mu_2=1$.
Moreover, in all cases  $$r_{1,2}+1>0 \ \ \mbox{and} \ \ \rho_{1,2}+1>0.$$

We may rewrite \eqref{U-Dyn} as
$$\Big(U'(t)+\frac{r_1}{1+t}U(t)\Big)'+\frac{r_2+1}{1+t}\Big(U'(t)+\frac{r_1}{1+t}U(t)\Big)=\int_{\mathbb{R}^n}|v(t,x)|^pdx.$$
Multiplying by $(1+t)^{r_2+1}$ and integrating over $[0,t]$, we obtain
$$(1+t)^{r_2+1}\Big(U'(t)+\frac{r_1}{1+t}U(t)\Big)-\Big(U'(0)+r_1U(0)\Big)=\int_0^t (1+s)^{r_2+1}\int_{\mathbb{R}^n}|v(s,x)|^pdx \,ds.$$
Using \eqref{assumptions u0, u1}, we have
\begin{equation}\label{key}
U'(t)+\frac{r_1}{1+t}U(t)>(1+t)^{-r_2-1}\int_0^t(1+s)^{r_2+1}\int_{\mathbb{R}^n}|v(s,x)|^pdx \, ds.
\end{equation}
Multiplying the above inequality by $(1+t)^{r_1}$ and integrating over $[0,t]$, we arrive at
\begin{equation*}
(1+t)^{r_1}U(t)-U(0)>\int_0^t(1+\tau)^{r_1-r_2-1}\int_0^\tau(1+s)^{r_2+1} \int_{\mathbb{R}^n}|v(s,x)|^pdx \, ds \, d\tau.
\end{equation*}
Since $u_0$ is nonnegative, we have
\begin{align}
 U(t)&\geqslant \int_0^t\bigg(\frac{1+\tau}{1+t}\bigg)^{r_1} \int_0^\tau \bigg(\frac{1+s}{1+\tau}\bigg)^{r_2+1} \int_{\mathbb{R}^n}|v(s,x)|^pdx \, ds\, d\tau.  \label{iter1}
\end{align}

 Furthermore, using H\"{o}lder inequality and the compactness of the support of solution with respect to $x$, we get from \eqref{iter1}
\begin{align}
  U(t) &\geqslant C_0\int_0^t\bigg(\frac{1+\tau}{1+t}\bigg)^{r_1}\int_0^\tau \bigg(\frac{1+s}{1+\tau}\bigg)^{r_2+1}(1+s)^{-n(p-1)}|V(s)|^p ds\, d\tau ,\label{iter2}
\end{align}
where $$C_0\doteq (\meas(B_1))^{1-p}R^{-n(p-1)}>0.$$

In a similar way, from \eqref{V-Dyn} we may derive
\begin{align}
V(t)&\geqslant \int_0^t\bigg(\frac{1+\tau}{1+t}\bigg)^{\rho_1} \int_0^\tau \bigg(\frac{1+s}{1+\tau}\bigg)^{\rho_2+1} \int_{\mathbb{R}^n}|u(s,x)|^q dx \, ds\, d\tau  \label{iter3} \\
 &\geqslant K_0\int_0^t\bigg(\frac{1+\tau}{1+t}\bigg)^{\rho_1}\int_0^\tau \bigg(\frac{1+s}{1+\tau}\bigg)^{\rho_2+1}(1+s)^{-n(q-1)}|U(s)|^q ds\, d\tau\label{iter4},
\end{align} where $$K_0\doteq(\meas(B_1))^{1-q}R^{-n(q-1)}>0.$$

\subsubsection*{Iteration argument}
Now we can proceed with the second step. We shall apply an iteration method based on lower bound estimates \eqref{Priori u^q}, \eqref{Priori v^p} and on the iteration frame \eqref{iter1}-\eqref{iter4}.
In comparison to the iteration method for a single semilinear wave equation with scale-invariant damping and mass (cf. \cite[Section 3]{PT18}), as the system is weakly coupled, we will combine the lower bounds for $U$ and $V$.

By using an induction argument, we will prove that 
\begin{align}
U(t)& \geqslant D_j(1+t)^{-a_j}(t-T_0)^{b_j} \quad \mbox{for} \ \ t\geqslant T_0, \label{lower bound U j} \\
V(t)& \geqslant \Delta_j(1+t)^{-\alpha_j}(t-T_0)^{\beta_j} \quad \mbox{for} \ \ t\geqslant T_0, \label{lower bound V j}
\end{align} where $\{a_j\}_{j\geqslant 1}$, $\{b_j\}_{j\geqslant 1}$, $\{D_j\}_{j\geqslant 1}$, $\{\alpha_j\}_{j\geqslant 1}$, $\{\beta_j\}_{j\geqslant 1}$ and $\{\Delta_j\}_{j\geqslant 1}$ are suitable sequences of positive real numbers that we shall determine throughout the iteration procedure.

Let us begin with the base case $j=1$ in \eqref{lower bound U j} and \eqref{lower bound V j}. Plugging \eqref{Priori v^p} in \eqref{iter1} and shrinking the domain of integration, we find for $t\geqslant T_0$
\begin{align*}
U(t) & \geqslant K_1\varepsilon^p (1+t)^{-r_1} \int_{T_0}^t (1+\tau)^{r_1-r_2-1}\int_{T_0}^\tau  (1+s)^{r_2+n-(n+\mu_2-1)\frac{p}{2}} \, ds\, d\tau \\
& \geqslant K_1\varepsilon^p (1+t)^{-r_1} \int_{T_0}^t (1+\tau)^{r_1-r_2-1-(n+\mu_2-1)\frac{p}{2}}\int_{T_0}^\tau  (1+s)^{r_2+n} \, ds\, d\tau \\
& \geqslant K_1\varepsilon^p (1+t)^{-r_2-1-(n+\mu_2-1)\frac{p}{2}} \int_{T_0}^t \int_{T_0}^\tau  (s-T_0)^{r_2+n} \, ds\, d\tau \\
& = K_1\varepsilon^p (1+t)^{-r_2-1-(n+\mu_2-1)\frac{p}{2}}  \frac{(t-T_0)^{r_2+n+2}}{(r_2+n+1)(r_2+n+2)},
\end{align*} which is the desired estimate, if we put
\begin{align*}
D_1 &\doteq \frac{ K_1\varepsilon^p}{(r_2+n+1)(r_2+n+2)},\qquad  a_1 \doteq r_2+1+(n+\mu_2-1)\frac{p}{2}, \quad b_1\doteq r_2+n+2.
\end{align*}
Analogously, we can prove \eqref{lower bound V j} for $j=1$ combining \eqref{iter3} and \eqref{Priori u^q}, provided that
\begin{align*}
\Delta_1 & \doteq \frac{ C_1\varepsilon^q}{(\rho_2+n+1)(\rho_2+n+2)}, \qquad \alpha_1 \doteq \rho_2+1+(n+\mu_1-1)\frac{q}{2}, \quad \beta_1\doteq \rho_2+n+2.
\end{align*}

Let us proceed with the inductive step: \eqref{lower bound U j} and \eqref{lower bound V j} are assumed to be true for $j\geqslant 1$, we prove them for $j+1$. Let us plug \eqref{lower bound V j} in \eqref{iter2}. Then, shrinking the domain of integration and using the positiveness of $\alpha_j$ and $\beta_j$ and the condition $r_2+1>0$, for $t\geqslant T_0$ we get
\begin{align*}
 U(t) &\geqslant C_0 \, \Delta_j^p  \, (1+t)^{-r_1} \int_{T_0}^t (1+\tau)^{r_1-r_2-1}\int_{T_0}^\tau  (1+s)^{r_2+1+n(1-p)-\alpha_j p}(s-T_0)^{\beta_j p} \, ds\, d\tau \\ 
 &\geqslant C_0 \, \Delta_j^p  \, (1+t)^{-r_1} \int_{T_0}^t (1+\tau)^{r_1-r_2-1-n(p-1)-\alpha_j p}\int_{T_0}^\tau  (s-T_0)^{r_2+1+\beta_j p} \, ds\, d\tau \\
 &\geqslant C_0 \, \Delta_j^p  \, (1+t)^{-r_2-1-n(p-1)-\alpha_j p} \int_{T_0}^t \int_{T_0}^\tau  (s-T_0)^{r_2+1+\beta_j p} \, ds\, d\tau \\
 & = \frac{C_0 \, \Delta_j^p }{(r_2+2+\beta_j p)(r_2+3+\beta_j p)} (1+t)^{-r_2-1-n(p-1)-\alpha_j p} (t-T_0)^{r_2+3+\beta_j p},
\end{align*} that is, \eqref{lower bound U j} for $j+1$ provided that
\begin{align*}
D_{j+1}\doteq \frac{C_0 \, \Delta_j^p }{(r_2+2+\beta_j p)(r_2+3+\beta_j p)}, \qquad a_{j+1}\doteq r_2+1+n(p-1)+\alpha_j p, \quad b_{j+1}\doteq r_2+3+\beta_j p. 
\end{align*}
Similarly, we can prove \eqref{lower bound V j} for $j+1$ combining \eqref{iter4} and \eqref{lower bound U j}, in the case in which
\begin{align*}
\Delta_{j+1}\doteq \frac{K_0 \, D_j^q }{(\rho_2+2+b_j q)(\rho_2+3+b_j q)} \qquad \alpha_{j+1}\doteq \rho_2+1+n(q-1)+a_j q, \quad \beta_{j+1}\doteq \rho_2+3+b_j q. 
\end{align*}

Let us determine explicitly the expression for $a_j,b_j,\alpha_j,\beta_j$ at least for odd $j$. Let us start with $a_j$. Using the previous relations, we have
\begin{align*}
a_j &= r_2+1+n(p-1)+\alpha_{j-1} \, p =  r_2+1+n(p-1)+\big(\rho_2+1+n(q-1)+a_{j-2}\, q\big) p \\
& = \underbrace{r_2+1-n+(\rho_2+1)p+npq}_{\doteq A}+pq \, a_{j-2}.
\end{align*} Applying iteratively the previous relation, for odd $j$ we get 
\begin{align}
a_j &= A + pq \, a_{j-2} = A + A \, pq \, a_{j-2}+ (pq)^2 a_{j-4} = \ \ \cdots \notag\\
&  = A \sum_{k=0}^{(j-3)/2} (pq)^k + (pq)^{\frac{j-1}{2}}a_1 = A \frac{(pq)^{\frac{j-1}{2}}-1}{pq-1} + (pq)^{\frac{j-1}{2}}a_1 \notag \\ &= \bigg(\frac{A}{pq-1}+a_1\bigg)(pq)^{\frac{j-1}{2}} -\frac{A}{pq-1}.   \label{explicit epression a j}
\end{align}

In a similar way, for odd $j$ we get
\begin{align}
\alpha_j = \bigg(\frac{\widetilde{A}}{pq-1}+\alpha_1\bigg)(pq)^{\frac{j-1}{2}} -\frac{\widetilde{A}}{pq-1},  \label{explicit epression alpha j}
\end{align} where $\widetilde{A}\doteq \rho_2+1-n+(r_2+1)q+npq$. 
For the sake of simplicity we do not derive the representations of $a_j$ and $\alpha_j$ for even $j$, as it is unnecessary to prove the theorem. 

Analogously, for odd $j$ we have, combining the definitions of $b_j$ and $\beta_j$,
\begin{align}
b_j & = r_2+3+\beta_{j-1} \,p = r_2+3+\big( \rho_2+3+b_{j-2}\, q\big) p = \underbrace{r_2+3+(\rho_2+3)p}_{\doteq B}+ pq \, b_{j-2},\label{iterative relation b j} \\
\beta_j & = \rho_2+3+b_{j-1} \,q = \rho_2+3+\big( r_2+3+\beta_{j-2}\, p\big) q = \underbrace{\rho_2+3+(r_2+3)q}_{\doteq \widetilde{B}}+ pq \, \beta_{j-2}. \label{iterative relation beta j}
\end{align} Also,
\begin{align}
b_j = \bigg(\frac{B}{pq-1}+b_1\bigg)(pq)^{\frac{j-1}{2}} -\frac{B}{pq-1} \ \ \mbox{and} \ \
\beta_j = \bigg(\frac{\widetilde{B}}{pq-1}+\beta_1\bigg)(pq)^{\frac{j-1}{2}} -\frac{\widetilde{B}}{pq-1}. \label{explicit epression b j and beta j}
\end{align}

The next step is to derive lower bounds for $D_j$ and $\Delta_j$. From  the definition of $D_j$ and $\Delta_j$ it follows immediately 
\begin{align}\label{lower bound Dj Deltaj n1}
D_{j} & \geqslant \frac{C_0  }{b_{j}^2}  \Delta_{j-1}^p \ \ \mbox{and} \ \ \
\Delta_{j}  \geqslant \frac{K_0 }{\beta_{j}^2} D_{j-1}^q .
\end{align} 
Therefore, the next step is to determine upper bounds for $b_j$ and for $\beta_j$, respectively. If $j$ is odd, plugging the first equation from \eqref{explicit epression b j and beta j} for $j-2$ in \eqref{iterative relation b j} and using the definition of $B$, it follows
\begin{align*}
b_j & =r_2+3+(\rho_2+3)p+ pq \, \bigg[\bigg(\frac{B}{pq-1}+b_1\bigg)(pq)^{\frac{j-3}{2}} -\frac{B}{pq-1}\bigg] \\
& =r_2+3+\bigg(\rho_2+3-\frac{Bq}{pq-1}\bigg)p+\bigg(\frac{B}{pq-1}+b_1\bigg)(pq)^{\frac{j-1}{2}} \\
& =-\frac{(r_2+3)+(\rho_2+3)p}{pq-1}+\bigg(\frac{B}{pq-1}+b_1\bigg)(pq)^{\frac{j-1}{2}} < B_0 \, (pq)^{\frac{j-1}{2}},
\end{align*} where $$B_0\doteq \frac{B}{pq-1}+b_1=n-1+\frac{(r_2+3)pq+(\rho_2+3)p}{pq-1}>0.$$
Similarly, for odd $j$, employing \eqref{iterative relation beta j} and the second equation in \eqref{explicit epression b j and beta j}, one finds
\begin{align*}
\beta_j < \widetilde{B}_0 \, (pq)^{\frac{j-1}{2}},
\end{align*} where $$\widetilde{B}_0\doteq \frac{\widetilde{B}}{pq-1}+\beta_1=n-1+\frac{(\rho_2+3)pq+(r_2+3)q}{pq-1}>0.$$
It is possible to derive similar estimates also for $b_{j-1}$ and $\beta_{j-1}$. Indeed, from \eqref{iterative relation b j} and \eqref{explicit epression b j and beta j} we get
\begin{align*}
b_{j-1}& =r_2+3+\beta_{j-2}\, p= r_2+3-\frac{\widetilde{B}p}{pq-1}+\frac{1}{q}\bigg(\frac{\widetilde{B}}{pq-1}+\beta_1\bigg)(pq)^{\frac{j-1}{2}} < \widetilde{B}_0 \, (pq)^{\frac{j-1}{2}}, \\
\beta_{j-1}& =\rho_2+3+b_{j-2}\, q= \rho_2+3-\frac{Bq}{pq-1}+\frac{1}{p}\bigg(\frac{B}{pq-1}+b_1\bigg)(pq)^{\frac{j-1}{2}} < B_0 \, (pq)^{\frac{j-1}{2}}.
\end{align*}
Hence, due to the above derived upper bounds for $b_j,b_{j-1},\beta_j,\beta_{j-1}$, from \eqref{lower bound Dj Deltaj n1} it follows
\begin{align} \label{lower bound Dj  n2}
D_j & \geqslant \frac{C_0}{B_0^2}\frac{\Delta_{j-1}^p}{(pq)^{j-1}} \geqslant \frac{C_0 K_0^p}{B_0^2} \frac{D_{j-2}^{pq}}{(pq)^{j-1}\beta_{j-1}^{2p}}  \geqslant 
\frac{\widetilde{C} D_{j-2}^{pq}}{\big((pq)^{p+1}\big)^{j-1}},\\  \label{lower bound Deltaj  n2}
\Delta_j & \geqslant \frac{K_0}{\widetilde{B}_0^2}\frac{D_{j-1}^q}{(pq)^{j-1}} \geqslant \frac{K_0 C_0^q}{\widetilde{B}_0^2} \frac{\Delta_{j-2}^{pq}}{(pq)^{j-1}b_{j-1}^{2q}}  \geqslant 
  \frac{\widetilde{K} \Delta_{j-2}^{pq}}{\big((pq)^{q+1}\big)^{j-1}},
\end{align} where $\widetilde{C}\doteq C_0 K_0^p/B_0^{2(p+1)}$ and $\widetilde{K}\doteq K_0 C_0^q/\widetilde{B}_0^{2(q+1)}$. 

From \eqref{lower bound Dj  n2}, if $j$ is odd, then, it follows
\begin{align*}
\log D_j & \geqslant pq \log D_{j-2} -(j-1)(p+1)\log (pq) +\log \widetilde{C} \\
& \geqslant (pq)^2 \log D_{j-4} -\big((j-1)+(j-3)pq\big)(p+1)\log (pq) +\big(1+pq\big)\log \widetilde{C} \\
 & \geqslant \ \cdots \\
 & \geqslant (pq)^{\frac{j-1}{2}} \log D_{1} -\Bigg(\sum_{k=1}^{(j-1)/2}(j+1-2k)\,(pq)^{k-1}\Bigg)(p+1)\log (pq) +\Bigg(\sum_{k=0}^{(j-3)/2}(pq)^k\Bigg)\log \widetilde{C}.
\end{align*}

Using an inductive argument, the following formulas can be shown:
\begin{align*}
\sum_{k=0}^{(j-3)/2}(pq)^k =\frac{(pq)^{\frac{j-1}{2}}-1}{pq-1}
\end{align*} and
\begin{align*}
\sum_{k=1}^{(j-1)/2}(j+1-2k)\, (pq)^{k-1}= \frac{1}{pq-1} \bigg(2(pq) \, \frac{(pq)^{\frac{j-1}{2}}-1}{pq-1}-j+1\bigg).
\end{align*}

Consequently,
\begin{align*}
\log D_j & \geqslant (pq)^{\frac{j-1}{2}} \bigg[\log D_{1}-\frac{2(pq)(p+1)}{(pq-1)^2}\log(pq)+\frac{\log \widetilde{C}}{pq-1}\bigg]+\frac{2(pq)(p+1)}{(pq-1)^2}\log(pq)\\ & \qquad +(j-1)\frac{(p+1)}{pq-1}\log(pq) -\frac{\log \widetilde{C}}{pq-1}.
\end{align*}
Thus, for an odd $j$ such that $j>\frac{\log \widetilde{C}}{(p+1)\log(pq)}-\frac{2(pq)}{pq-1}+1$, it holds
\begin{align} \label{lower bound log Dj}
\log D_j\geqslant (pq)^{\frac{j-1}{2}}\big(\log D_1-S_{p,q}(\infty)\big),
\end{align} where $S_{p,q}(\infty)\doteq \frac{2(pq)(p+1)}{(pq-1)^2}\log(pq)-\frac{\log \widetilde{C}}{pq-1}$.

In a similar way, one can show for an odd $j$ the validity of 
\begin{align*}
\log \Delta_j & \geqslant (pq)^{\frac{j-1}{2}} \bigg[\log \Delta_{1}-\frac{2(pq)(q+1)}{(pq-1)^2}\log(pq)+\frac{\log \widetilde{K}}{pq-1}\bigg]+\frac{2(pq)(q+1)}{(pq-1)^2}\log(pq)\\ & \qquad +(j-1)\frac{(q+1)}{pq-1}\log(pq) -\frac{\log \widetilde{K}}{pq-1},
\end{align*} and, then, for $j>\frac{\log \widetilde{K}}{(q+1)\log(pq)}-\frac{2(pq)}{pq-1}+1$ this yields 
\begin{align}\label{lower bound log Deltaj}
\log \Delta_j\geqslant (pq)^{\frac{j-1}{2}}\big(\log \Delta_1-\widetilde{S}_{p,q}(\infty)\big),
\end{align} where $\widetilde{S}_{p,q}(\infty)\doteq \frac{2(pq)(q+1)}{(pq-1)^2}\log(pq)-\frac{\log \widetilde{K}}{pq-1}$.\newline For the sake of brevity, we denote $j_0\doteq \big\lceil  \frac{1}{ \log(pq)} \max\{\frac{\widetilde{C}}{p+1},\frac{\widetilde{K}}{q+1}\}-\frac{2pq}{pq-1}+1\big\rceil$.

Let us combine now \eqref{lower bound U j} and \eqref{lower bound log Dj}. For an odd $j>j_0$ and $t\geqslant T_0$, using \eqref{explicit epression a j} and  \eqref{explicit epression b j and beta j}, we get
\begin{align*}
U(t) & \geqslant \exp\Big((pq)^{\frac{j-1}{2}}\big(\log D_1- S_{p,q}(\infty)\big)\Big) (1+t)^{-a_j}(t-T_0)^{b_j} \\
& = \exp\Big((pq)^{\frac{j-1}{2}}\big(\log D_1- S_{p,q}(\infty)\big)\Big) (1+t)^{-\big(\frac{A}{pq-1}+a_1\big)(pq)^{\frac{j-1}{2}} +\frac{A}{pq-1}}(t-T_0)^{\big(\frac{B}{pq-1}+b_1\big)(pq)^{\frac{j-1}{2}} -\frac{B}{pq-1}} \\
& = \exp\Big((pq)^{\frac{j-1}{2}}\big(\log D_1-\big(\tfrac{A}{pq-1}+a_1\big)\log(1+t)+\big(\tfrac{B}{pq-1}+b_1\big)\log(t-T_0)- S_{p,q}(\infty)\big)\Big) \\ & \qquad \quad\times (1+t)^{\frac{A}{pq-1}}(t-T_0)^{-\frac{B}{pq-1}}.
\end{align*}
Also, for $t\geqslant 2T_0+1$ from the previous estimate it follows
\begin{align}
U(t) & \geqslant \exp\Big((pq)^{\frac{j-1}{2}}J(t)\Big)  (1+t)^{\frac{A}{pq-1}}(t-T_0)^{-\frac{B}{pq-1}}, \label{lower bound U with J}
\end{align} where 
\begin{align*}
J(t) & \doteq \log D_1+\big(\tfrac{B-A}{pq-1}+b_1-a_1\big)\log(t-T_0)-\big(\tfrac{A}{pq-1}+a_1\big)\log 2- S_{p,q}(\infty) \\
& =  \log\Big( D_1(t-T_0)^{\tfrac{B-A}{pq-1}+b_1-a_1}\Big)-\big(\tfrac{A}{pq-1}+a_1\big)\log 2- S_{p,q}(\infty).
\end{align*}
Let us calculate more precisely the power of $(t-T_0)$ in the last line:
\begin{align*}
\frac{B-A}{pq-1}+b_1-a_1& = \frac{2p+2+n-npq}{pq-1}+n+1-(n+\mu_2-1)\frac{p}{2} \\ & =\frac{pq+1+2p}{pq-1}-(n+\mu_2-1)\frac{p}{2} = p\, \bigg(\frac{q+p^{-1}+2}{pq-1}- \frac{n+\mu_2-1}{2}\bigg).
\end{align*} So, $\frac{B-A}{pq-1}+b_1-a_1 >0$ if and only if $ F(n+\mu_2,q,p) >0$.

In an analogous way, from \eqref{lower bound V j}, \eqref{lower bound log Deltaj}, \eqref{explicit epression alpha j} and \eqref{explicit epression b j and beta j} we obtain for $t\geqslant 2T_0+1$ and for an odd $j>j_0$
\begin{align}
V(t) & \geqslant \exp\Big((pq)^{\frac{j-1}{2}}\widetilde{J}(t)\Big)  (1+t)^{\tfrac{\widetilde{A}}{pq-1}}(t-T_0)^{-\tfrac{\widetilde{B}}{pq-1}}, \label{lower bound V with Jtilde}
\end{align} where 
\begin{align*}
\widetilde{J}(t) & \doteq \log \Delta_1+\big(\tfrac{\widetilde{B}-\widetilde{A}}{pq-1}+\beta_1-\alpha_1\big)\log(t-T_0)-\big(\tfrac{\widetilde{A}}{pq-1}+\alpha_1\big)\log 2- \widetilde{S}_{p,q}(\infty) \\
& =  \log\Big( \Delta_1(t-T_0)^{\tfrac{\widetilde{B}-\widetilde{A}}{pq-1}+\beta_1-\alpha_1}\Big)-\big(\tfrac{\widetilde{A}}{pq-1}+\alpha_1\big)\log 2- \widetilde{S}_{p,q}(\infty).
\end{align*} In this case, $\tfrac{\widetilde{B}-\widetilde{A}}{pq-1}+\beta_1-\alpha_1=q\Big(\frac{p+2+q^{-1}}{pq-1}-\frac{n+\mu_1-1}{2}\Big)>0$ if and only if $F(n+\mu_1,p,q) >0$.

If $F(n+\mu_2,q,p)>0$, since $D_1= K_2 \varepsilon^p$, where $K_2\doteq  K_1(r_2+n+1)^{-1}(r_2+n+2)^{-1}$, then, $J(t)>0$ is equivalent to require
\begin{align*}
 t>T_0+E\varepsilon^{-F(n+\mu_2,q,p)^{-1}}, \qquad \mbox{where} \ \ E \doteq \left(e^{\big(\tfrac{A}{pq-1}+a_1\big)\log 2+ S_{p,q}(\infty)}K_2^{-1}\right)^{\frac{1}{p F(n+\mu_2,q,p)}}.
\end{align*} 
If we choose $\varepsilon_0>0$ sufficiently small so that $$2E\varepsilon_0^{-F(n+\mu_2,q,p)^{-1}}>2T_0+1,$$ then, for any $\varepsilon\in (0,\varepsilon_0]$ and $t>2E\varepsilon^{-F(n+\mu_2,q,p)^{-1}}$ we have $t> 2T_0+1$ and  $J(t)>0$. 
Thus, letting $j\to \infty$ in \eqref{lower bound U with J}, the lower bound for $U$ blows up and, hence, $U$ can be finite only for $t\leqslant 2E\varepsilon^{-F(n+\mu_2,q,p)^{-1}}$.

Analogously, in the case $F(n+\mu_1,p,q)>0$, as $\Delta_1= C_2 \varepsilon^q$, where $C_2\doteq  C_1(\rho_2+n+1)^{-1}(\rho_2+n+2)^{-1}$, we get that $\widetilde{J}(t)>0$ is equivalent to
\begin{align*}
t> T_0+\widetilde{E}\varepsilon^{-F(n+\mu_1,p,q)^{-1}}, \qquad \mbox{where} \ \ \widetilde{E} \doteq \bigg(e^{\big(\tfrac{\widetilde{A}}{pq-1}+\alpha_1\big)\log 2+ \widetilde{S}_{p,q}(\infty)}C_2^{-1}\bigg)^{\frac{1}{q F(n+\mu_1,p,q)}}.
\end{align*} 
Also,  in this case we may choose $\varepsilon_0>0$ sufficiently small so that $$2\widetilde{E}\varepsilon_0^{-F(n+\mu_1,p,q)^{-1}}>2T_0+1.$$ Consequently, for any $\varepsilon\in (0,\varepsilon_0]$ and $t>2\widetilde{E}\varepsilon^{-F(n+\mu_1,p,q)^{-1}}$ we have $t> 2T_0+1$ and  $\widetilde{J}(t)>0$ and, then, taking the limit as $j\to \infty$ in \eqref{lower bound V with Jtilde} the lower bound for $V(t)$  diverges. Hence, $V$ may be finite just for $t\leqslant2\widetilde{E}\varepsilon^{-F(n+\mu_1,p,q)^{-1}}$. 
Summarizing, we proved that if \eqref{critical exponent wave like case system} holds, then, $(u,v)$ blows up in finite time and \eqref{lifespan upper bound estimate} is satisfied. 
This completes the proof.

\section{Super-solutions of the scale-invariant wave equations and their properties} \label{Section supersol}

Henceforth we deal with the critical case and the proof of Theorem \ref{Thm critical case}.
In this section we introduce the notion of super-solutions of the Cauchy problem
\begin{align}\label{scale inv equation H}
\begin{cases}
u_{tt}-\Delta u +\frac{\mu}{1+t}u_t +\frac{\nu^2}{(1+t)^2}u = H,  & x\in \mathbb{R}^n, \ t\in (0,T),  \\
u(0,x)= \varepsilon u_0(x), & x\in \mathbb{R}^n,  \\
u_t(0,x)= \varepsilon u_1(x),& x\in \mathbb{R}^n, 
\end{cases}
\end{align}
and, then, we derive some estimates related to super-solutions.

\begin{defn}\label{def 3.1} Let $(u_0,u_1)\in H^1(\mathbb{R}^n)\times L^2(\mathbb{R}^n)$ be compactly supported and $H\in L^1_{\loc}([0,T)\times \mathbb{R}^n)$. We say that $u\in \mathcal{C}([0,T),H^1(\mathbb{R}^n))\cap \mathcal{C}^1([0,T),L^2(\mathbb{R}^n))$ is a super-solution of \eqref{scale inv equation H} on $[0,T)$ if
$u(0,x)=\varepsilon u_0(x)$ in $H^1(\mathbb{R}^n)$ 
and 
\begin{align}
&\varepsilon \int_{\mathbb{R}^n} u_1(x) \Psi(0,x) \, dx+\int_0^T \int_{\mathbb{R}^n} H(t,x) \Psi(t,x) \, dx \, dt \leqslant \int_0^T \int_{\mathbb{R}^n} -u_t(t,x) \Psi_t(t,x) \, dx \, dt \notag \\ & \quad + \int_0^T \int_{\mathbb{R}^n} \Big(\nabla u(t,x) \cdot \nabla \Psi(t,x)+\frac{\mu}{1+t}u_t(t,x)\Psi(t,x) +\frac{\nu^2}{(1+t)^2}u(t,x)\Psi(t,x) \Big)\, dx \, dt \label{integral relation supersol}
\end{align} for any nonnegative test function $\Psi\in \mathcal{C}_0^\infty([0,T)\times \mathbb{R}^n)$.
\end{defn}

\begin{lem}\label{lemma 3.3} Let $u$ be a super-solution of \eqref{scale inv equation H} with $(u_0,u_1)$ in the classical energy space $H^1(\mathbb{R}^n)\times L^2(\mathbb{R}^n)$ and compactly supported in $B_{r_0}$, $H$ locally summable and $\supp u\subset \{(t,x)\in [0,T)\times\mathbb{R}^n:|x|\leqslant r_0+t\}$. Then, it holds
\begin{align}
& \varepsilon \int_{\mathbb{R}^n} \big( u_1(x) \Psi(0,x)+u_0(x) (\mu\Psi(0,x)-\Psi_t(0,x) \big) \, dx+\int_0^T \int_{\mathbb{R}^n} H(t,x) \Psi(t,x) \, dx \, dt \notag\\  & \quad  \leqslant \int_0^T \int_{\mathbb{R}^n} u(t,x)\Big( \Psi_{tt}(t,x) -\Delta \Psi(t,x)-\frac{\mu}{1+t}\Psi_t(t,x) +\frac{\mu+\nu^2}{(1+t)^2}\Psi(t,x) \Big)\, dx \, dt 
\label{integral relation supersol, 2 deriv t}
\end{align} for any nonnegative test function $\Psi\in \mathcal{C}_0^\infty([0,T)\times \mathbb{R}^n)$.
\end{lem}

\begin{proof} 
Using the support condition for $u$, integration by parts provides 
\begin{align*}
& \int_0^T \int_{\mathbb{R}^n} \Big( -u_t(t,x) \Psi_t(t,x) +\nabla u(t,x) \cdot \nabla \Psi(t,x)+\frac{\mu}{1+t}u_t(t,x)\Psi(t,x) \Big)\, dx \, dt \\ 
& \quad = \int_0^T \frac{d}{dt} \bigg[\int_{\mathbb{R}^n} (-u(t,x) \Psi_t(t,x)+\frac{\mu}{1+t}u(t,x)\Psi(t,x)) \, dx \bigg] dt \\ & \qquad + \int_0^T \int_{\mathbb{R}^n} u(t,x)\Big( \Psi_{tt}(t,x) -\Delta \Psi(t,x)-\partial_t \Big(\frac{\mu}{1+t}\Psi(t,x)\Big) \Big)\, dx \, dt \\
& \quad = \varepsilon \int_{\mathbb{R}^n} u_0(x) \big( \Psi_t(0,x)-\mu \Psi(0,x)\big) \, dx   + \int_0^T \int_{\mathbb{R}^n} u(t,x)\Big( \Psi_{tt}(t,x) -\Delta \Psi(t,x)-\partial_t \Big(\frac{\mu}{1+t}\Psi(t,x)\Big) \Big)\, dx \, dt.
\end{align*} Substituting this relation in \eqref{integral relation supersol}, we get \eqref{integral relation supersol, 2 deriv t}. This concludes the proof.
\end{proof}

In the next result, we will employ the following solution of the adjoint equation to the homogeneous linear equation related to \eqref{scale inv equation H}, which is a particular solution among the self-similar solutions that we will introduce in Section \ref{Section self similar sol}:
\begin{align}
V(t,x)& \doteq (1+t)^{\frac{\mu+1+\sqrt{\delta}}{2}}((1+t)^2-|x|^2)^{-\frac{n+\sqrt{\delta}}{2}} \notag\\ &=(1+t)^{-n+\frac{\mu+1-\sqrt{\delta}}{2}}\bigg(1-\frac{|x|^2}{(1+t)^2}\bigg)^{-\frac{n+\sqrt{\delta}}{2}} \qquad  \mbox{for} \ \ (t,x)\in Q, \label{def V}
\end{align} where $Q \doteq \{(t,x)\in [0,T)\times\mathbb{R}^n:|x|\leqslant 1+t\}$.

Moreover, we introduce a parameter dependent bump function. Let $\psi\in \mathcal{C}^\infty_0([0,\infty))$ be a nonincreasing function such that $\psi=1$ on $[0,\frac{1}{2}]$ and $\supp \psi \subset [0,1)$. Besides, we denote $$\psi^*(t)=\begin{cases} 0 & \mbox{if} \ \ t\in[0,\frac{1}{2}), \\ \psi(t) & \mbox{if} \ \ t\in[\frac{1}{2},\infty). \end{cases}$$ Clearly, $\psi^*$ is not smooth. We will use this bounded function only to keep trace of the support property of derivatives of $\psi$. More precisely, if $\psi_R(t)\doteq \psi(\tfrac{t}{R}),\psi_R^*(t)\doteq \psi^*(\tfrac{t}{R})$ for any $t\geqslant 0$ with $R>0$, then, the following estimates hold (see, for example Lemma 3.1 in \cite{ISW18})
\begin{align}\label{estimate psi' R}
| \partial_t \psi_R(t)| \lesssim R^{-1} [\psi_R^*(t)]^{1-\frac{1}{k}} \qquad \mbox{for any} \ \ k\geqslant1, \\
| \partial_t^2 \psi_R(t)| \lesssim R^{-2} [\psi_R^*(t)]^{1-\frac{2}{k}} \qquad \mbox{for any} \ \ k\geqslant 2. \label{estimate psi'' R}
\end{align}

Now we can prove a lower bound estimate, which is somehow related to \eqref{Priori u^q} and \eqref{Priori v^p}.

\begin{lem}\label{lemma 3.4} Let $u_0\in H^1(\mathbb{R}^n)$ and $u_1\in L^2(\mathbb{R}^n)$ be nonnegative functions such that $\supp u_0,\supp u_1 \subset B_{r_0}$ with $r_0\in (0,1)$.
Let $u$ be a super-solution of 
\begin{align} \label{scale inv equation homogeneous case}
\begin{cases}
u_{tt}-\Delta u +\frac{\mu}{1+t}u_t +\frac{\nu^2}{(1+t)^2}u = 0,  & x\in \mathbb{R}^n, \ t\in (0,T),  \\
u(0,x)= \varepsilon u_0(x), & x\in \mathbb{R}^n,  \\
u_t(0,x)= \varepsilon u_1(x),& x\in \mathbb{R}^n, 
\end{cases}
\end{align} such that $\supp u\subset Q_{r_0}\doteq \{(t,x)\in [0,T)\times\mathbb{R}^n:|x|\leqslant r_0+t\}$. 
Then, for any $p>1$ and any $R\in (1,T)$ it holds 
\begin{align}\label{lower bound int u^p psi*R}
\big(I_{\mu,\nu^2}[u_0,u_1]\varepsilon\big)^p R^{n-\frac{n+\mu-1}{2}p} \lesssim \int_0^T \int_{\mathbb{R}^n}|u(t,x)|^p \psi_R^*(t) \,dx \, dt,
\end{align} where the multiplicative constant in \eqref{lower bound int u^p psi*R} is independent of $\varepsilon$ and $R$ and $$I_{\mu,\nu^2}[u_0,u_1] \doteq \int_{\mathbb{R}^n}\Big(u_1(x)V(0,x)+u_0(x)\big(\mu V(0,x)-V_t(0,x)\big)\Big)\, dx.$$
\end{lem}

\begin{rem} In the previous statement the nonnegativity of $u_0,u_1$ can be relaxed by requiring simply that $u_0,u_1$ satisfy $I_{\mu,\nu^2}[u_0,u_1] >0$.
\end{rem}

\begin{proof}
Let us consider $\Psi(t,x)=\psi_R(t)V(t,x)\chi (x)$, where $\chi\in \mathcal{C}^\infty_0(\mathbb{R}^n)$ satisfies $\chi=1$ on $B_{r_0+T}$. Applying \eqref{integral relation supersol, 2 deriv t} to this $\Psi$, we get
\begin{align*}
& \varepsilon \int_{\mathbb{R}^n} \big( u_1(x) \psi_R(0) V(0,x)+u_0(x) (\mu\psi_R(0)V(0,x)-\psi'_R(0)V(0,x)-\psi_R(0)V_t(0,x) \big) \, dx  \\  &  \leqslant \int_0^T \int_{\mathbb{R}^n} u(t,x)\Big( \partial_t^2 (\psi_R(t)V(t,x))-\Delta (\psi_R(t)V(t,x))-\tfrac{\mu}{1+t}\partial_t (\psi_R(t)V(t,x)) +\tfrac{\mu+\nu^2}{(1+t)^2}\psi_R(t)V(t,x) \Big)\, dx \, dt 
\end{align*} and, then,
\begin{align*}
 \varepsilon \int_{\mathbb{R}^n} \big( u_1(x)& V(0,x)  +u_0(x) (\mu V(0,x)-V_t(0,x) \big) \, dx = \varepsilon \, I_{\mu,\nu^2}[u_0,u_1] \\  &  \leqslant \int_0^T \int_{\mathbb{R}^n} u(t,x)\Big( \partial_t^2 \psi_R(t)V(t,x)+2\partial_t \psi_R(t)V_t(t,x)-\tfrac{\mu}{1+t}\partial_t\psi_R(t)V(t,x) \Big)\, dx \, dt  \\
& \qquad + \int_0^T \int_{\mathbb{R}^n} u(t,x)\psi_R(t)\Big(V_{tt}(t,x)-\Delta V(t,x)-\tfrac{\mu}{1+t} V_t(t,x) +\tfrac{\mu+\nu^2}{(1+t)^2}V(t,x) \Big)\, dx \, dt 
\\  &  = \int_0^T \int_{\mathbb{R}^n} u(t,x)\Big( \partial_t^2 \psi_R(t)V(t,x)+2\partial_t \psi_R(t)V_t(t,x)-\tfrac{\mu}{1+t}\partial_t\psi_R(t)V(t,x) \Big)\, dx \, dt ,
\end{align*} where in last step we used the fact that $V$ solves the adjoint equation of the homogeneous wave equation with scale-invariant damping and mass.

Let us remark that
\begin{align*}
V_t(t,x)=-(1+t)^{-n-\frac{\mu-1-\sqrt{\delta}}{2}}\Big(1-\tfrac{|x|^2}{(1+t)^2}\Big)^{-\frac{n+\sqrt{\delta}}{2}-1}\Big(n+\tfrac{\sqrt{\delta}-(\mu+1)}{2}+\tfrac{\mu+1+\sqrt{\delta}}{2}\tfrac{|x|^2}{(1+t)^2}\Big),
\end{align*} so that
\begin{align*}
\mu V(0,x)-V_t(0,x)= (1-|x|^2)^{-\frac{n+\sqrt{\delta}}{2}-1}\Big(n+\tfrac{\mu-1+\sqrt{\delta}}{2}+\tfrac{1-\mu+\sqrt{\delta}}{2}|x|^2\Big)\geqslant 0 \qquad \mbox{ for any} \ \ x\in B_{r_0}.
\end{align*} In particular, for nonnegative and nontrivial $u_0,u_1$ the last estimate yields $I_{\mu,\nu^2}[u_0,u_1] >0$. 

If we employ now \eqref{estimate psi' R} and \eqref{estimate psi'' R} for $k=p'$ and $k=2p'$, respectively, then, we arrive at
\begin{align*}
\varepsilon \, I_{\mu,\nu^2}[u_0,u_1] &\lesssim  \int_0^T \int_{\mathbb{R}^n} |u(t,x)|\Big( \tfrac{|V(t,x)|}{R^2}+\tfrac{|V_t(t,x)|}{R}+\tfrac{\mu}{1+t}\tfrac{|V(t,x)|}{R} \Big) [\psi^*_R(t)]^{\frac{1}{p}}\, dx \, dt \\
&  \lesssim  \int_0^T \int_{\mathbb{R}^n} |u(t,x)|\Big( \tfrac{|V(t,x)|}{R^2}+\tfrac{|V_t(t,x)|}{R}\Big) [\psi^*_R(t)]^{\frac{1}{p}}\, dx \, dt \\
&  \lesssim  \bigg(\int_0^T \int_{\mathbb{R}^n} |u(t,x)|^p\, \psi^*_R(t)\, dx \, dt \bigg)^{\frac{1}{p}} \bigg(\int_{\frac{R}{2}}^R \int_{B_{r_0+t}}\Big( \tfrac{|V(t,x)|}{R^2}+\tfrac{|V_t(t,x)|}{R}\Big)^{p'} \, dx \, dt \bigg)^{\frac{1}{p'}}.
\end{align*}
For $t\in[\frac{R}{2},R]$ and $|x|\leqslant r_0+t$ it holds
\begin{align*}
|V(t,x)| & \lesssim R^{-n+\frac{\mu+1-\sqrt{\delta}}{2}}\Big(1-\tfrac{|x|}{1+t}\Big)^{-\frac{n+\sqrt{\delta}}{2}}, \qquad
|V_t(t,x)|  \lesssim R^{-n+\frac{\mu-1-\sqrt{\delta}}{2}}\Big(1-\tfrac{|x|}{1+t}\Big)^{-\frac{n+\sqrt{\delta}}{2}-1}.
\end{align*} Therefore,
\begin{align}
\varepsilon  I_{\mu,\nu^2}[u_0,u_1] & \lesssim R^{-n-2+\frac{\mu+1-\sqrt{\delta}}{2}}\bigg(\int_0^T \! \int_{\mathbb{R}^n} \! |u(t,x)|^p\, \psi^*_R(t)\, dx \, dt \bigg)^{\frac{1}{p}} \bigg(\int_{\frac{R}{2}}^R \!\int_{B_{r_0+t}}\!\!\Big(1-\tfrac{|x|}{1+t}\Big)^{-\big(\frac{n+\sqrt{\delta}}{2}+1\big)p'} dx \, dt \bigg)^{\frac{1}{p'}} \notag \\
& \lesssim R^{-\frac{n}{p}+\frac{n+\mu-1}{2}}\bigg(\int_0^T \! \int_{\mathbb{R}^n} \! |u(t,x)|^p\, \psi^*_R(t)\, dx \, dt \bigg)^{\frac{1}{p}}, \label{intermediate int}
\end{align} where in the second inequality we used
\begin{align*}
\int_{\frac{R}{2}}^R \!\int_{B_{r_0+t}}\!\!\Big(1-\tfrac{|x|}{1+t}\Big)^{-\big(\frac{n+\sqrt{\delta}}{2}+1\big)p'} dx \, dt & \lesssim  \int_{\frac{R}{2}}^R (1+t)^{\big(\frac{n+\sqrt{\delta}}{2}+1\big)p'}(r_0+t)^{n-1}\!\int_0^{r_0+t}\!(1+t-r)^{-\big(\frac{n+\sqrt{\delta}}{2}+1\big)p'} dr \, dt \\
& \lesssim  \int_{\frac{R}{2}}^R (1+t)^{\big(\frac{n+\sqrt{\delta}}{2}+1\big)p'}(r_0+t)^{n-1} \, dt \lesssim R^{n+\big(\frac{n+\sqrt{\delta}}{2}+1\big)p'}.
\end{align*}
From \eqref{intermediate int} it follows easily \eqref{lower bound int u^p psi*R}. The proof is complete.
\end{proof}

\section{Self-similar solutions related to Gauss hypergeometric functions} \label{Section self similar sol}

In the critical case of blow-up phenomena for semilinear wave equations with scale-invariant damping and mass, it is important to have a precise description of the behavior of solutions to the adjoint equation to the corresponding linear homogeneous equation. According to this purpose, in this section we will introduce a family of self-similar solutions to this equation, that can be represented by using Gauss hypergeometric functions (see also \cite{Zhou07,ZH14,IS17s,IS17,ISW18,PT18}). In particular, we refer to 
  \cite[Section 4]{PT18} for the proofs of results which are not proved here.
  
Hence, our goal is to provide a family of solutions on $Q$ to the adjoint equation
\begin{align}\label{adjoint equation}
\partial_t^2 \Phi -\Delta \Phi -\partial_t \left(\frac{\mu }{1+t} \, \Phi\right)+ \frac{\nu^2}{(1+t)^2}\,\Phi =0.
\end{align}
Let $\beta$ be a real parameter. If we make the following ansatz:
\begin{align*}
\Phi_\beta(t,x)\doteq (1+t)^{-\beta+1} \phi_{\beta}\bigg(\frac{|x|^2}{(1+t)^2}\bigg),
\end{align*} where $\psi_\beta\in \mathcal{C}^2([0,1))$, then, $\Phi_\beta$ solves \eqref{adjoint equation} if and only if $\phi_\beta$ solves
\begin{align}\label{phi equation}
z(1-z) \phi''_\beta (z)+\big(\tfrac{n}{2}-\big(\beta+\tfrac{\mu+1}{2}\big)z\big)\phi'_\beta(z)-\tfrac{1}{4}\big(\beta(\beta+\mu-1)+\nu^2\big)\, \phi_\beta(z)=0.
\end{align}
  
Choosing 
\begin{align*}
a_\beta & =a_\beta(\mu,\nu^2)\doteq \tfrac{\beta}{2}+\tfrac{\mu-1}{4}+\tfrac{\sqrt{\delta}}{4}  \qquad \mbox{and} \qquad
b_\beta  =b_\beta(\mu,\nu^2)\doteq  \tfrac{\beta}{2}+\tfrac{\mu-1}{4}-\tfrac{\sqrt{\delta}}{4} ,
\end{align*} we have 
\begin{align*}
a_\beta+b_\beta+1=\beta+\tfrac{\mu+1}{2} \qquad \mbox{and} \qquad a_\beta b_\beta =\tfrac{1}{4}\big(\beta(\beta+\mu-1)+\nu^2\big).
\end{align*} Therefore, \eqref{phi equation} coincides with the hypergeometric equation with parameters  $(a_\beta,b_\beta; \tfrac{n}{2})$, namely,
\begin{align*}
z(1-z) \phi''_\beta (z)+\big(\tfrac{n}{2}-\big(a_\beta+b_\beta+1\big)z\big)\phi'_\beta(z)-a_\beta b_\beta \, \phi_\beta(z)=0.
\end{align*}
Also, we may choose $\phi_\beta$ as the Gauss hypergeometric function
\begin{align*}
\phi_{\beta}(z)\doteq  \mathsf{F}(a_\beta,b_\beta;\tfrac{n}{2};z) = \sum_{k=0}^\infty \frac{(a_\beta)_k \, (b_\beta)_k}{(n/2)_k} \frac{z^k}{k!} \qquad |z|<1,
\end{align*} where $(m)_k$ denotes Pochhammer's symbol, which is  defined by $$(m)_k\doteq \begin{cases} 1 & \mbox{if} \ \ k=0, \\ \prod_{j=1}^k (m+j-1)  & \mbox{if} \ \ k>0.\end{cases}$$

\begin{defn} Let $\beta$ a real parameter such that $\beta>\frac{\sqrt{\delta}+1-\mu}{2}$. Then, we define 
\begin{align}
\Phi_{\beta,\mu,\nu^2}(t,x) & \doteq (1+t)^{-\beta+1} \phi_\beta \Big(\tfrac{|x|^2}{(1+t)^2}\Big) \notag \\ & = (1+t)^{-\beta+1} \mathsf{F}\Big(a_\beta (\mu,\nu^2)\, ,b_\beta (\mu,\nu^2)\, ;\tfrac{n}{2}\, ;\tfrac{|x|^2}{(1+t)^2}\Big) \qquad \mbox{for} \ \ (t,x)\in Q. \label{def Phi beta}
\end{align}
\end{defn}

 According to the construction we explained until now in this section, it is clear that $\{\Phi_{\beta,\mu,\nu^2}\}_{\beta}$  is a family of solutions to \eqref{adjoint equation}.
 In the next lemma, we discuss some properties of this family of self-similar solutions.
 
\begin{lem} \label{lemma 3.6} The function $\Phi_{\beta,\mu,\nu^2}$ satisfies the following properties:
\begin{enumerate}
\item[\rm{(i)}] $\Phi_{\beta,\mu,\nu^2}$ is a solution of \eqref{adjoint hom system} on $Q$.
\item[\rm{(ii)}] $|\partial_t \Phi_{\beta,\mu,\nu^2} | \lesssim  \Phi_{\beta+1,\mu,\nu^2}$ on $Q$. 
\item[\rm{(iii)}] If $\beta\in \big(\frac{\sqrt{\delta}+1-\mu}{2},\frac{n+1-\mu}{2}\big)$, then, $$\Phi_{\beta,\mu,\nu^2}(t,x) \approx (1+t)^{-\beta+1}$$ for any $ (t,x)\in Q$.
\item[\rm{(iv)}] If $\beta>\frac{n+1-\mu}{2}$, then, $$\Phi_{\beta,\mu,\nu^2}(t,x) \approx (1+t)^{-\beta+1}\Big(1-\tfrac{|x|^2}{(1+t)^2}\Big)^{\frac{n-\mu+1}{2}-\beta} $$ for any  $ (t,x)\in Q$.
\end{enumerate}
\end{lem}

\begin{proof}
Let us prove (ii). If we denote $z\doteq \frac{|x|^2}{(1+t)^2}$, then,
\begin{align}
\partial_t \Phi_{\beta,\mu,\nu^2}(t,x) & =(1+t)^{-\beta} \big[(1-\beta)\mathsf{F}(a_\beta,b_\beta; \tfrac{n}{2}; z)-2z \, \mathsf{F}'(a_\beta,b_\beta; \tfrac{n}{2}; z)\big]\notag \\
& =(1+t)^{-\beta} \big[(1-\beta)\mathsf{F}(a_\beta,b_\beta; \tfrac{n}{2}; z)-4 \tfrac{a_\beta b_\beta}{n}z \, \mathsf{F}(a_\beta +1,b_\beta +1; \tfrac{n}{2}+1; z)\big]. \label{Phi beta der t}
\end{align} Moreover, 
\begin{align*}
\Phi_{\beta+1,\mu,\nu^2}(t,x) & =(1+t)^{-\beta} \mathsf{F}(a_{\beta+1},b_{\beta+1}; \tfrac{n}{2}; z) = (1+t)^{-\beta} \mathsf{F}(a_{\beta}+\tfrac{1}{2},b_{\beta}+\tfrac{1}{2}; \tfrac{n}{2}; z).
\end{align*}
Since for real parameters $(a,b;c)$ the hypergeometric function has the following behavior for $z\in [0,1)$ 
\begin{align} \label{hypergeometric function asymptotic}
\mathsf{F}(a,b;c;z) \approx \begin{cases} 1 & \mbox{if} \ \ c>a+b, \\ -\log(1-z) & \mbox{if} \ \ c=a+b, \\ (1-z)^{c-(a+b)} & \mbox{if} \ \ c<a+b,\end{cases}
\end{align} and $\frac{n}{2}+1 -(a_\beta+b_\beta+2)=\frac{n}{2}-(a_{\beta+1}+b_{\beta+1})$, as the second term in \eqref{Phi beta der t} is the dominant one, we get immediately the desired property. By using \eqref{hypergeometric function asymptotic}, we find (iii) and (iv) as well.
\end{proof}

\begin{rem} If we consider $\beta$ such that $b_\beta = \frac{n}{2}$, i.e. $\beta=n+\frac{\sqrt{\delta}+1-\mu}{2}$, then, $a_\beta=b_\beta+\frac{\sqrt{\delta}}{2}$ and
\begin{align*}
\Phi_{\beta,\mu,\nu}(t,x)=(1+t)^{-n+\frac{\mu+1-\sqrt{\delta}}{2}} \mathsf{F}\Big(\tfrac{n+\sqrt{\delta}}{2},\tfrac{n}{2};\tfrac{n}{2}; \tfrac{|x|^2}{(1+t)^2}\Big)= V(t,x)
\end{align*} with $V$ defined by \eqref{def V}. In the previous equality, we used the relation $\mathsf{F}(\alpha,\gamma;\gamma;z)=(1-z)^{-\alpha}$.
\end{rem}

\begin{lem}\label{lemma 3.7}
Let us assume $u_0 \in H^1(\mathbb{R}^n)$ and $u_1\in L^2(\mathbb{R}^n)$ satisfying $\supp u_0,\supp u_1 \subset B_{r_0}$ for some $r_0\in (0,1)$ and 
\begin{align*}
J_{\beta,\mu,\nu^2}[u_0,u_1] \doteq \int_{\mathbb{R}^n}\Big(u_1(x)\Phi_{\beta,\mu,\nu^2}(0,x)+u_0(x)\big(\mu \Phi_{\beta,\mu,\nu^2}(0,x)-\partial_t\Phi_{\beta,\mu,\nu^2}(0,x)\big)\Big) \, dx >0,
\end{align*} where $\beta>\frac{\sqrt{\delta}+1-\mu}{2}$ is a parameter.
Let $u$ be a super-solution of \eqref{scale inv equation H}
such that $\supp u\subset Q_{r_0}.$ 

Then, for any $p>1$ and any $R\in (1,T)$ it holds  
\begin{align}
&J_{\beta,\mu,\nu^2}[u_0,u_1]\,\varepsilon+ \int_0^T \int_{\mathbb{R}^n} H(t,x) \, \psi_R(t) \, \Phi_{\beta,\mu,\nu^2}(t,x) \, dx \, dt  \notag \\ & \quad \lesssim R^{-1} \int_0^T \int_{\mathbb{R}^n}|u(t,x)| \, \Phi_{\beta+1,\mu,\nu^2}(t,x) \,  [\psi_R^*(t)]^{\frac{1}{p}} \,dx \, dt, \label{lower bound int u Phi beta+1 psi*R 1/p}
\end{align} where the multiplicative constant in \eqref{lower bound int u Phi beta+1 psi*R 1/p} is independent of $\varepsilon$ and $R$.
\end{lem}

\begin{proof}
Let us consider $\Psi(t,x)=\psi_R(t)\Phi_{\beta,\mu,\nu^2}(t,x)\chi (x)$, where $\chi\in \mathcal{C}^\infty_0(\mathbb{R}^n)$ satisfies $\chi=1$ on $B_{r_0+T}$. Applying \eqref{integral relation supersol, 2 deriv t} to the test function $\Psi$, we get
\begin{align}
 \varepsilon J_{\beta,\mu,\nu^2}&[u_0,u_1]  + \int_0^T\!\int_{\mathbb{R}^n}  H(t,x) \, \psi_R(t)\, \Phi_{\beta,\mu,\nu^2}(t,x)\, dx \, dt \notag \\  &  \leqslant \int_0^T \int_{\mathbb{R}^n} u(t,x)\Big( \partial_t^2 \psi_R(t)\Phi_{\beta,\mu,\nu^2}(t,x)+2\partial_t \psi_R(t)\partial_t\Phi_{\beta,\mu,\nu^2}(t,x)-\tfrac{\mu}{1+t}\partial_t\psi_R(t)\Phi_{\beta,\mu,\nu^2}(t,x) \Big)\, dx \, dt \notag \\
& \qquad + \int_0^T \int_{\mathbb{R}^n} u(t,x)\psi_R(t)\Big(\partial_t^2-\Delta-\tfrac{\mu}{1+t} \partial_t  +\tfrac{\mu+\nu^2}{(1+t)^2} \Big) \Phi_{\beta,\mu,\nu^2}(t,x)\, dx \, dt \notag \\
& \lesssim \int_0^T \int_{\mathbb{R}^n} |u(t,x)|\Big( R^{-2}|\Phi_{\beta,\mu,\nu^2}(t,x)|+R^{-1}|\partial_t\Phi_{\beta,\mu,\nu^2}(t,x)|\Big)[\psi^*_R(t)]^{\frac{1}{p}}\, dx \, dt,\label{intermediate estimate int J beta}
\end{align} where in last inequality we used the fact that $\Phi_{\beta,\mu,\nu^2}$ solves \eqref{adjoint equation} and \eqref{estimate psi' R}, \eqref{estimate psi'' R}.
We note that for $t\in[\frac{R}{2},R]$ and $|x|\leqslant r_0+t$, it holds
\begin{align*}
\Phi_{\beta,\mu,\nu^2}(t,x) & =(1+t)^{-\beta +1} \mathsf{F}\Big(\tfrac{\beta}{2}+\tfrac{\mu-1+\sqrt{\delta}}{4},\tfrac{\beta}{2}+\tfrac{\mu-1-\sqrt{\delta}}{4}; \tfrac{n}{2};\tfrac{|x|^2}{(1+t)^2}\Big) \\
& \lesssim R (1+t)^{-\beta } \mathsf{F}\Big(\tfrac{\beta}{2}+\tfrac{\mu-1+\sqrt{\delta}}{4},\tfrac{\beta}{2}+\tfrac{\mu-1-\sqrt{\delta}}{4}; \tfrac{n}{2};\tfrac{|x|^2}{(1+t)^2}\Big) \\
& \lesssim R (1+t)^{-\beta } \mathsf{F}\Big(\tfrac{\beta+1}{2}+\tfrac{\mu-1+\sqrt{\delta}}{4},\tfrac{\beta+1}{2}+\tfrac{\mu-1-\sqrt{\delta}}{4}; \tfrac{n}{2};\tfrac{|x|^2}{(1+t)^2}\Big) = R \, \Phi_{\beta+1,\mu,\nu^2}(t,x) 
\end{align*} and, then, combining the previous estimate with Lemma \ref{lemma 3.6} (ii), we get $$ R^{-2}|\Phi_{\beta,\mu,\nu^2}(t,x)|+R^{-1}|\partial_t\Phi_{\beta,\mu,\nu^2}(t,x)|\lesssim R^{-1} \Phi_{\beta+1,\mu,\nu^2}(t,x).$$ Thus, if we use the last estimate in the right hand side of \eqref{intermediate estimate int J beta} we get \eqref{lower bound int u Phi beta+1 psi*R 1/p}. This completes the proof.
\end{proof}

\begin{rem} Let us rewrite the function that multiplies $u_0$ in $J_{\beta,\mu,\nu^2}[u_0,u_1]$ in a more explicit way:
\begin{align*}
\mu \Phi_{\beta,\mu,\nu^2}(0,x)-\partial_t\Phi_{\beta,\mu,\nu^2}(0,x)= (\beta+\mu-1)\mathsf{F}\big(a_\beta,b_\beta; \tfrac{n}{2}; |x|^2\big)+\frac{4 a_\beta b_\beta}{n}\mathsf{F}\big(a_{\beta+1},b_{\beta+1}; \tfrac{n+1}{2}; |x|^2\big).
\end{align*} Then, if $\beta\geqslant 1-\mu$, in order to get a strictly positive $J_{\beta,\mu,\nu^2}[u_0,u_1]$, it is sufficient to consider nonnegative and nontrivial $u_0,u_1$. Since in our treatment  either $\beta\in \big(\tfrac{\sqrt{\delta}+1-\mu}{2},\tfrac{n-\mu+1}{2}\big)$ or  $\beta\geqslant \tfrac{n-\mu+1}{2}$, we may assume without loss of regularity that $\beta\geqslant 1-\mu$ thanks to $1-\mu< \frac{n-\mu+1}{2}$.
\end{rem}

\begin{lem}\label{lemma 3.8}
Let $\beta>\frac{\sqrt{\delta}+1-\mu}{2}$ be a real number such that $\beta\neq \frac{n-\mu+1}{2}$. Then, the following estimate holds for $R\geqslant R_0>0$:
\begin{align*}
\int_{\frac{R}{2}}^R \int_{B_{r_0+t}} (\Phi_{\beta,\mu,\nu^2}(t,x))^{p'} \, dx \, dt \lesssim \begin{cases} R^{n+1+(1-\beta)p'} & \mbox{if} \ \ \beta< \frac{n-\mu+1}{2}+1-\frac{1}{p}, \\ R^{-\frac{n-\mu-1}{2}p'+n} \log R & \mbox{if} \ \ \beta= \frac{n-\mu+1}{2}+1-\frac{1}{p}, \\ R^{-\frac{n-\mu-1}{2}p'+n}  & \mbox{if} \ \ \beta> \frac{n-\mu+1}{2}+1-\frac{1}{p} .\end{cases}
\end{align*}
\end{lem}

\begin{proof}
Let us begin with the case $\beta< \frac{n-\mu+1}{2}$. Using Lemma \ref{lemma 3.6} (iii), we get
\begin{align*}
\int_{\frac{R}{2}}^R \int_{B_{r_0+t}} (\Phi_{\beta,\mu,\nu^2}(t,x))^{p'} \, dx \, dt & \approx \int_{\frac{R}{2}}^R \int_{B_{r_0+t}} (1+t)^{(-\beta+1)p'} \, dx \, dt \\ & \approx \int_{\frac{R}{2}}^R (1+t)^{(-\beta+1)p'} (r_0+t)^n \, dt \lesssim R^{n+1+(-\beta+1)p'}.
\end{align*} When $\beta> \frac{n-\mu+1}{2}$,  from Lemma \eqref{lemma 3.6} (iv) it follows
\begin{align*}
\int_{\frac{R}{2}}^R \int_{B_{r_0+t}} (\Phi_{\beta,\mu,\nu^2}(t,x))^{p'}  \, dx \, dt 
 & \lesssim \int_{\frac{R}{2}}^R  (1+t)^{(-\beta+1)p'}  \int_0^{r_0+t}\Big(1-\tfrac{r}{(1+t)}\Big)^{\frac{n-\mu+1}{2}p'-\beta p'} r^{n-1}\, dr \, dt \\
  & \lesssim \int_{\frac{R}{2}}^R  (1+t)^{-\frac{n-\mu-1}{2}p'} (r_0+t)^{n-1} \int_0^{r_0+t}(1+t-r)^{\frac{n-\mu+1}{2}p'-\beta p'}\, dr \, dt \\
  & \lesssim \begin{cases}  \vphantom{\Big(} \int_{\frac{R}{2}}^R  (1+t)^{-\frac{n-\mu-1}{2}p'} (r_0+t)^{n-1}\, dt  & \mbox{if} \ \ \beta> \frac{n-\mu+1}{2}+1-\frac{1}{p},\\ \vphantom{\Big(} \int_{\frac{R}{2}}^R  (1+t)^{-\frac{n-\mu-1}{2}p'} (r_0+t)^{n-1} \log(1+t)\, dt  & \mbox{if} \ \  \beta = \frac{n-\mu+1}{2}+1-\frac{1}{p}, \\ \vphantom{\Big(} \int_{\frac{R}{2}}^R  (1+t)^{(-\beta+1) p'+1} (r_0+t)^{n-1} \, dt & \mbox{if} \ \ \beta< \frac{n-\mu+1}{2}+1-\frac{1}{p} \end{cases} \\
  & \lesssim \begin{cases} R^{-\frac{n-\mu-1}{2}p'+n}  & \mbox{if} \ \ \beta> \frac{n-\mu+1}{2}+1-\frac{1}{p},\\ R^{-\frac{n-\mu-1}{2}p'+n} \log R & \mbox{if} \ \  \beta = \frac{n-\mu+1}{2}+1-\frac{1}{p}, \\ R^{n+1+(-\beta+1)p'} & \mbox{if} \ \ \beta< \frac{n-\mu+1}{2}+1-\frac{1}{p}. \end{cases}
\end{align*} Combining the two cases, we find the desired estimate.
\end{proof}

\section{Critical case: Proof of Theorem \ref{Thm critical case}} \label{Section critical case}

This section is organized as follows: firstly, we recall some technical lemmas from \cite{IS18,ISW18}; then, in the last two subsections we prove the blow-up results and the corresponding upper bounds for the lifespan in the critical case $p=p_0(n+\mu)$ for \eqref{scale inv eq} and on the critical curve $\max\{F(n+\mu_1,p,q);F(n+\mu_2,q,p)\}=0$ for \eqref{weakly coupled system}, respectively. 

\subsection{Lemmas on the blow-up dynamic in critical cases}

The results stated in this section are already know in the literature (see \cite{IS18,ISW18}). Nonetheless, for the ease of the reader they will be recalled.  The upcoming lemmas will play a fundamental role in determining the upper bound lifespan estimate of exponential type, whenever we are in a critical case.

\begin{defn} Let $w\in L^1_{\loc}([0,T), L^1(\mathbb{R}^n))$ be a nonnegative function. We set
\begin{align*}
Y[w](R)\doteq \int_0^R \bigg(\int_0^T \int_{\mathbb{R}^n} w(t,x) \,\psi_\sigma^*(t) \, dx \, dt \bigg) \sigma^{-1} \, d\sigma, \qquad \mbox{for any} \ \ R\in (0,T).
\end{align*} 
\end{defn}

The functional $Y[w]$ satisfies the properties stated in the next lemma.

\begin{lem}\label{lemma 3.9} Let $w\in L^1_{\loc}([0,T), L^1(\mathbb{R}^n))$ be a nonnegative function. Then, $Y[w]\in \mathcal{C}^1((0,T))$ and for any $R\in (0,T)$
\begin{align*}
\frac{d}{dR}Y[w](R) & = R^{-1} \int_0^T \int_{\mathbb{R}^n} w(t,x) \,\psi_R^*(t) \, dx \, dt,  \\
Y[w](R) & \leqslant \int_0^T \int_{\mathbb{R}^n} w(t,x) \,\psi_R(t) \, dx \, dt.
\end{align*}
\end{lem}

For the proof of the above lemma, one can see \cite[Proposition 2.1]{IS18}.

\begin{lem}\label{lemma 3.10} Let $y\in \mathcal{C}^1([R_0,T))$ be a nonnegative function, where $2<R_0<T$. Moreover, there exist $\theta, K_1,K_2>0$ and $p_1,p_2>1$ such that
\begin{align*}
\begin{cases}
 R \,y'(R) \geqslant K_1 \theta  & \mbox{for} \ \ R\in (R_0,T), \\
R \,(\log R)^{p_2-1}  y'(R) \geqslant K_2 (y(R))^{p_1}  & \mbox{for} \ \ R\in (R_0,T).
\end{cases}
\end{align*} If $p_2<p_1+1$, then, there exist $\theta_0,K>0$ such that
\begin{align*}
T\leqslant \exp\big(K \theta^{-\frac{p_1-1}{p_1-p_2+1}}\big) \qquad \mbox{for any} \ \ \theta\in (0,\theta_0).
\end{align*}
\end{lem}

See \cite[Lemma 3.10]{ISW18} for the proof of Lemma \ref{lemma 3.10}.

\subsection{Critical case for the single semilinear equation}

In this section we derive upper bound estimates for the lifespan of super-solutions of the semilinear wave equation with scale-invariant damping and mass in the critical case. Even though the result has been already proved for solutions in \cite[Theorem 1.3]{PT18}, we need to use this generalization to super-solutions in Section \ref{Subsection critical weak coupled system}.

Let us introduce the notion of super solutions for the semilinear model.

\begin{defn}\label{def 4.1} Let $c,\varepsilon$ be positive real constants. Let $(u_0,u_1)\in H^1(\mathbb{R}^n)\times L^2(\mathbb{R}^n)$ be compactly supported and $p>1$. We say that $u\in \mathcal{C}([0,T),H^1(\mathbb{R}^n))\cap \mathcal{C}^1([0,T),L^2(\mathbb{R}^n))\cap L^p_{\loc}([0,T)\times \mathbb{R}^n)$ is a super-solution on $[0,T)$ of
\begin{align}\label{scale inv equation semi c}
\begin{cases}
u_{tt}-\Delta u +\frac{\mu}{1+t}u_t +\frac{\nu^2}{(1+t)^2}u = c|u|^p,  & x\in \mathbb{R}^n, \ t\in (0,T),  \\
u(0,x)= \varepsilon u_0(x), & x\in \mathbb{R}^n,  \\
u_t(0,x)= \varepsilon u_1(x),& x\in \mathbb{R}^n, 
\end{cases}
\end{align}if
$u(0,x)=\varepsilon u_0(x)$ in $H^1(\mathbb{R}^n)$ 
and 
\begin{align}
&\varepsilon \int_{\mathbb{R}^n} u_1(x) \Psi(0,x) \, dx+\int_0^T \int_{\mathbb{R}^n} c\,  |u(t,x)|^p \Psi(t,x) \, dx \, dt \leqslant \int_0^T \int_{\mathbb{R}^n} -u_t(t,x) \Psi_t(t,x) \, dx \, dt \notag \\ & \quad + \int_0^T \int_{\mathbb{R}^n} \Big(\nabla u(t,x) \cdot \nabla \Psi(t,x)+\frac{\mu}{1+t}u_t(t,x)\Psi(t,x) +\frac{\nu^2}{(1+t)^2}u(t,x)\Psi(t,x) \Big)\, dx \, dt \label{integral relation supersol semi eq c}
\end{align} for any nonnegative test function $\Psi\in \mathcal{C}_0^\infty([0,T)\times \mathbb{R}^n)$.
\end{defn}

\begin{prop}\label{Prop in the crit case semi eq}
 Let $u_0\in H^1(\mathbb{R}^n)$ and $ u_1\in L^2(\mathbb{R}^n)$ be nonnegative and compactly supported functions such that $\supp u_0,\supp u_1 \subset B_{r_0}$ with $r_0\in (0,1)$.
Let $p=p_0(n+\mu)$ and let $u$ be a super-solution of \eqref{scale inv equation semi c} on $[0,T(\varepsilon))$ such that $\supp u\subset Q_{r_0}$. Then, there exist two positive and independent of $\varepsilon$ constants $\varepsilon_0,C$ such that
\begin{align*}
T(\varepsilon)\leqslant \exp(C \varepsilon^{-p(p-1)}) \qquad \mbox{for any} \ \ \varepsilon \in (0,\varepsilon_0).
\end{align*} 
\end{prop}

\begin{proof}
In order to prove the proposition, we consider $Y=Y[|u|^p\Phi_{\beta_p,\mu,\nu^2}]$ with $\beta_p=\frac{n-\mu+1}{2}-\frac{1}{p}$. \newline
Since $p=p_0(n+\mu)$, then, 
\begin{align}\label{beta p condition}
-\beta_p+1+n-\tfrac{n+\mu-1}{2}p= \tfrac{1}{p}\big(1+\tfrac{n+\mu+1}{2}p-\tfrac{n+\mu-1}{2}p^2\big)=0.
\end{align} Therefore, using \eqref{lower bound int u^p psi*R}, we find
\begin{align*}
Y'(R)R & =\int_0^T\int_{\mathbb{R}^n} |u(t,x)|^p \, \Phi_{\beta_p,\mu,\nu^2}(t,x) \,\psi_R^*(t) \, dx \, dt 
\gtrsim R^{-\beta_p+1}\int_0^T\int_{\mathbb{R}^n} |u(t,x)|^p \, \psi_R^*(t) \, dx \, dt \\ & \gtrsim R^{-\beta_p+1+n-\frac{n+\mu-1}{2}p} \varepsilon^p= \varepsilon^p.
\end{align*}
Furthermore, from Lemmas \ref{lemma 3.9}, \ref{lemma 3.7} and \ref{lemma 3.8} we obtain
\begin{align*}
(Y(R))^p & \lesssim \bigg(\int_0^T\int_{\mathbb{R}^n} |u(t,x)|^p \, \Phi_{\beta_p,\mu,\nu^2}(t,x) \,\psi_R(t) \, dx \, dt \bigg)^p \\
& \lesssim R^{-p} \bigg(\int_0^T\int_{\mathbb{R}^n} |u(t,x)| \, \Phi_{\beta_p+1,\mu,\nu^2}(t,x) \,[\psi_R^*(t)]^{\frac{1}{p}} \, dx \, dt \bigg)^p \\
& \lesssim R^{-p} \bigg(\int_{\frac{R}{2}}^R \int_{B_{r_0+t}} (\Phi_{\beta_p+1,\mu,\nu^2}(t,x) )^{p'} \, dx \, dt \bigg)^{p-1} \int_0^T\int_{\mathbb{R}^n} |u(t,x)|^p \,\psi_R^*(t)\, dx \, dt  \\
& \lesssim R^{-p-\frac{n-\mu-1}{2}p+n(p-1)} (\log R)^{p-1} \int_0^T\int_{\mathbb{R}^n} |u(t,x)|^p \,\psi_R^*(t)\, dx \, dt \\
& \lesssim R^{\frac{n+\mu-1}{2}p-n} (\log R)^{p-1} R^{\beta_p-1} \int_0^T\int_{\mathbb{R}^n} |u(t,x)|^p\, \Phi_{\beta_p,\mu,\nu^2}(t,x) \,\psi_R^*(t)\, dx \, dt \\ &= (\log R)^{p-1}  \int_0^T\int_{\mathbb{R}^n} |u(t,x)|^p\, \Phi_{\beta_p,\mu,\nu^2}(t,x) \,\psi_R^*(t)\, dx \, dt = (\log R)^{p-1} Y'(R) R
\end{align*} where in the last inequality we employed \eqref{beta p condition}. Setting $\theta=\varepsilon^p$, $p_1=p_2=p$, by  Lemma \ref{lemma 3.10} it follows the upper bound for the lifespan $T(\varepsilon)\leqslant \exp\big(C \varepsilon^{- p(p-1)}\big)$ for a suitable constant $C$. 
\end{proof}

\subsection{Critical case for the weakly coupled system} \label{Subsection critical weak coupled system}

This section is devoted to the proof of Theorem \ref{Thm critical case}, but, before proving it, we will derive some estimates for the weakly coupled system \eqref{weakly coupled system} in the general case.

Let $(u,v)$ be a solution to \eqref{weakly coupled system} in the sense of Definition \ref{def energ sol intro}. 
As the nonlinear terms in \eqref{weakly coupled system} are nonnegative, in particular, $u,v$ are super-solutions of \eqref{scale inv equation homogeneous case} for $(\mu,\nu^2)=(\mu_j,\nu^2_j)$, $j=1,2$, respectively. Moreover, $\supp u, \supp v \subset Q_{r_0}$ due to the property of finite speed of propagation for hyperbolic equations.

Therefore, Lemma \ref{lemma 3.4} implies
\begin{align}
\int_0^T \int_{\mathbb{R}^n} |u(t,x)|^q \psi_R^*(t) \, dx \, dt \gtrsim \big( I_{\mu_1,\nu_1^2}[u_0,u_1]\varepsilon\big)^q R^{n-\frac{n+\mu_1-1}{2}q},  \label{(Zwei)} \\
\int_0^T \int_{\mathbb{R}^n} |v(t,x)|^p \psi_R^*(t) \, dx \, dt \gtrsim \big( I_{\mu_2,\nu_2^2}[v_0,v_1]\varepsilon\big)^p R^{n-\frac{n+\mu_2-1}{2}p}.  \label{(Eins)}
\end{align}
 
 Let us consider $\hat{\beta}\in (\frac{\sqrt{\delta_2}+1-\mu_2}{2},\frac{n-\mu_2+1}{2}-\frac{1}{p})$.  Note that the nonemptiness of the interval for $\hat{\beta}$ is guaranteed by the first condition in \eqref{technical restrictions on p,q critical case}. Using Lemma \ref{lemma 3.6} (iii), Lemma \ref{lemma 3.7} and Lemma \ref{lemma 3.8}, we get
 \begin{align*}
 \int_0^T \int_{\mathbb{R}^n} & |u(t,x)|^q \, \psi_R(t) \, dx \, dt \\  &\lesssim R^{\hat{\beta}-1} \int_0^T \int_{\mathbb{R}^n} |u(t,x)|^q \,  \Phi_{\hat{\beta},\mu_2,\nu_2^2}(t,x) \, \psi_R(t)  \, dx \, dt \\
 &\lesssim R^{\hat{\beta}-2} \int_0^T \int_{\mathbb{R}^n} |v(t,x)| \,  \Phi_{\hat{\beta}+1,\mu_2,\nu_2^2}(t,x) \, [\psi_R^*(t)]^{\frac{1}{p}}  \, dx \, dt \\
 &\lesssim R^{\hat{\beta}-2} \bigg(\int_0^T \int_{\mathbb{R}^n} \! |v(t,x)|^p \, \psi_R^*(t)  \, dx \, dt \bigg)^{\frac{1}{p}} \! \bigg(\int_{\frac{R}{2}}^R \int_{B_{r_0+t}}  \!\! \big(\Phi_{\hat{\beta}+1,\mu_2,\nu_2^2}(t,x)\big)^{p'}  \, dx \, dt \bigg)^{\frac{1}{p'}} \\
 &\lesssim
R^{-2+\frac{n+1}{p'}} \bigg(\int_0^T \int_{\mathbb{R}^n}  |v(t,x)|^p \, \psi_R^*(t)  \, dx \, dt \bigg)^{\frac{1}{p}}.
 \end{align*}
Raising to the $p$ power  both sides of the last inequality, we obtain
\begin{align} \label{(Drei)}
\bigg(\int_0^T \int_{\mathbb{R}^n} |u(t,x)|^q \, \psi_R(t) \, dx \, dt\bigg)^p \lesssim R^{-2+(n-1)(p-1)} \int_0^T \int_{\mathbb{R}^n}  |v(t,x)|^p \, \psi_R^*(t)  \, dx \, dt .
\end{align}
In a similar way, choosing $\tilde{\beta}\in (\frac{\sqrt{\delta_1}+1-\mu_1}{2},\frac{n-\mu_1+1}{2}-\frac{1}{q})$ and using Lemma \ref{lemma 3.6} (iii) to estimate $\Phi_{\tilde{\beta},\mu_1,\nu_1^2}$, one can prove
\begin{align}\label{(Vier)}
\bigg(\int_0^T \int_{\mathbb{R}^n} |v(t,x)|^p \, \psi_R(t) \, dx \, dt\bigg)^q \lesssim R^{-2+(n-1)(q-1)} \int_0^T \int_{\mathbb{R}^n}  |u(t,x)|^q \, \psi_R^*(t)  \, dx \, dt .
\end{align}

\begin{rem}\label{rem subcrit} Using \eqref{(Zwei)}, \eqref{(Eins)}, \eqref{(Drei)} and \eqref{(Vier)}, it is possible to prove the blow-up result and the corresponding upper bound for the lifespan in the subcritical case  in a simpler way than using the iteration argument. Nonetheless, additional technical restrictions on $(p,q)$, namely \eqref{technical restrictions on p,q critical case}, have to be considered, making the result obtained with the iteration argument sharper.
 
Indeed, combining \eqref{(Drei)} and \eqref{(Vier)} and the trivial inequality $\psi_R^*\leqslant \psi_R$, we find
\begin{align*}
\bigg(\int_0^T \int_{\mathbb{R}^n} |u(t,x)|^q \, \psi_R(t) \, dx \, dt\bigg)^{pq} & \lesssim R^{[-2+(n-1)(p-1)]q} \bigg(\int_0^T \int_{\mathbb{R}^n}  |v(t,x)|^p \, \psi_R^*(t)  \, dx \, dt\bigg)^q \\
&\lesssim R^{[-2+(n-1)(p-1)]q-2+(n-1)(q-1)}  \int_0^T \int_{\mathbb{R}^n}  |u(t,x)|^q \, \psi_R^*(t)  \, dx \, dt.
\end{align*} Rearranging the previous estimate, we arrive at
\begin{align*}
\bigg(\int_0^T \int_{\mathbb{R}^n} |u(t,x)|^q \, \psi_R(t) \, dx \, dt\bigg)^{pq-1} \lesssim R^{-2(q+1)+(n-1)(pq-1)},
\end{align*} that implies in turn
\begin{align*}
\int_0^T \int_{\mathbb{R}^n} |u(t,x)|^q \, \psi_R(t) \, dx \, dt\, \lesssim \, R^{-\frac{2(q+1)}{pq-1}+n-1}=R^{n-\frac{pq+2q+1}{pq-1}}.
\end{align*} Combining the previous inequality with \eqref{(Zwei)}, in the case $F(n+\mu_1,p,q)>0$ it follows
\begin{align*}
\varepsilon^q R^{n-\frac{n+\mu_1-1}{2}q} \lesssim \int_0^T \int_{\mathbb{R}^n} |u(t,x)|^q \psi_R^*(t) \, dx \, dt \lesssim R^{n-\frac{pq+2q+1}{pq-1}}. 
\end{align*} Comparing the lower bound and the upper bound for the integral in the last estimate, we obtain $$ R^{\frac{pq+2q+1}{pq-1}-\frac{n+\mu_1-1}{2}q}\lesssim \varepsilon^{-q},$$ which implies $R\lesssim \varepsilon^{-F(n+\mu_1,p,q)^{-1}} $.
Analogously,
\begin{align*}
\bigg(\int_0^T \int_{\mathbb{R}^n} |v(t,x)|^p \, \psi_R(t) \, dx \, dt\bigg)^{pq} & \lesssim R^{[-2+(n-1)(q-1)]p} \bigg(\int_0^T \int_{\mathbb{R}^n}  |u(t,x)|^q \, \psi_R^*(t)  \, dx \, dt\bigg)^p \\
&\lesssim R^{[-2+(n-1)(q-1)]p-2+(n-1)(p-1)}  \int_0^T \int_{\mathbb{R}^n}  |v(t,x)|^p \, \psi_R^*(t)  \, dx \, dt.
\end{align*} and \eqref{(Eins)} imply
\begin{align*}
\varepsilon^p R^{n-\frac{n+\mu_2-1}{2}p} \lesssim \int_0^T \int_{\mathbb{R}^n} |v(t,x)|^p \, \psi_R(t) \, dx \, dt \lesssim R^{n-\frac{pq+2p+1}{pq-1}}.
\end{align*} Proceeding as in the previous case, we have $R \lesssim \varepsilon^{-F(n+\mu_2,q,p)^{-1}}$ in the case $F(n+\mu_2,q,p)>0$. Therefore, letting $R\to T$, we obtained \eqref{lifespan upper bound estimate} provided that $p,q>1$ satisfy $\max\{F(n+\mu_1,p,q),F(n+\mu_2,q,p)\}>0$.
\end{rem}

By using the estimates that we proved in this section, we can now prove Theorem \ref{Thm critical case}. We will consider four subcases as in \eqref{lifespan upper bound estimate, critical case}.

\subsubsection{Case $F(n+\mu_1,p,q)=0>F(n+\mu_2,q,p)$} \label{Subsubsection 1}

Differently from the treatment of the subcritical case (cf. Remark \ref{rem subcrit}), in this case we study the blow-up dynamic of the function $Y=Y\big[|v|^p\Phi_{\beta_q,\mu_1,\nu_1^2}\, \big]$, where $\beta_q\doteq \frac{n-\mu_1+1}{2}-\frac{1}{q}\,$.

By \eqref{(Drei)} and \eqref{(Zwei)}, we get
\begin{align*}
  \int_0^T \int_{\mathbb{R}^n}  |v(t,x)|^p \, \psi_R^*(t)  \, dx \, dt  &\gtrsim R^{2-(n-1)(p-1)}\bigg(\int_0^T \int_{\mathbb{R}^n} |u(t,x)|^q \, \psi_R(t) \, dx \, dt\bigg)^p \\
 &\gtrsim R^{2-(n-1)(p-1)+np-\frac{n+\mu_1-1}{2}pq} \, \varepsilon^{pq} = R^{\frac{n-\mu_1-1}{2}-\frac{1}{q}} \, \varepsilon^{pq},
\end{align*} where in the last step we employed $F(n+\mu_1,p,q)=0$. Due to the definition of $\beta_q$, the last estimates implies
\begin{align}
 Y'(R)R & =\int_0^T \int_{\mathbb{R}^n}  |v(t,x)|^p \, \Phi_{\beta_q,\mu_1,\nu_1^2}(t,x) \, \psi_R^*(t)  \, dx \, dt  \gtrsim R^{-\beta_q+1}\int_0^T \int_{\mathbb{R}^n}  |v(t,x)|^p \, \psi_R^*(t)  \, dx \, dt  \gtrsim \varepsilon^{pq}. \label{Y lower bound}
\end{align}

By Lemma \ref{lemma 3.7} and Lemma \ref{lemma 3.8} in the logarithmic case, we find
\begin{align*}
 \int_0^T \int_{\mathbb{R}^n} |v(t,x)|^p\, & \Phi_{\beta_q,\mu_1,\nu_1^2}(t,x) \, \psi_R(t) \, dx \, dt \\ & \quad \lesssim R^{-1} \int_0^T \int_{\mathbb{R}^n} |u(t,x)|\, \Phi_{\beta_q+1,\mu_1,\nu_1^2}(t,x) \, [\psi_R^*(t)]^{\frac{1}{q}} \, dx \, dt \\
& \quad \lesssim R^{-1} \bigg(\int_0^T \int_{\mathbb{R}^n} |u(t,x)|^q \, \psi_R^*(t) \, dx \, dt\bigg)^{\frac{1}{q}}  \bigg(\int_{\frac{R}{2}}^R \int_{B_{r_0+t}} \big(\Phi_{\beta_q+1,\mu_1,\nu_1^2}(t,x)\Big)^{q'} \, dx \, dt\bigg)^{\frac{1}{q'}} \\
& \quad \lesssim R^{-1-\frac{n-\mu_1-1}{2}+\frac{n}{q'}} (\log R)^{\frac{1}{q'}} \bigg(\int_0^T \int_{\mathbb{R}^n} |u(t,x)|^q \, \psi_R^*(t) \, dx \, dt\bigg)^{\frac{1}{q}}  .
\end{align*}
Raising to the $pq$ power both sides of the previous inequality and using \eqref{(Drei)} and again the condition $F(n+\mu_1,p,q)=0$, we obtain
\begin{align}
\bigg( \int_0^T \int_{\mathbb{R}^n} |v(t,x)|^p\,  \Phi_{\beta_q,\mu_1,\nu_1^2}&(t,x) \, \psi_R(t) \, dx \, dt\bigg)^{pq} \notag \\ &\lesssim R^{\frac{n+\mu_1-1}{2}pq-np}(\log R)^{p(q-1)} \bigg(\int_0^T \int_{\mathbb{R}^n} |u(t,x)|^q \, \psi_R^*(t) \, dx \, dt\bigg)^p \notag \\
 & \lesssim  R^{\frac{n+\mu_1-1}{2}pq-np-2+(n-1)(p-1)}(\log R)^{p(q-1)}\int_0^T \int_{\mathbb{R}^n}  |v(t,x)|^p \, \psi_R^*(t)  \, dx \, dt  \notag \\
 & = R^{1-\beta_q}(\log R)^{p(q-1)}\int_0^T \int_{\mathbb{R}^n}  |v(t,x)|^p \, \psi_R^*(t)  \, dx \, dt \notag \\
 & \lesssim  (\log R)^{p(q-1)}\int_0^T \int_{\mathbb{R}^n}  |v(t,x)|^p \, \Phi_{\beta_q,\mu_1,\nu_1^2}(t,x)\, \psi_R^*(t)  \, dx \, dt. \label{Y log estimate}
\end{align}

Thanks to Lemma \ref{lemma 3.9}, from the last inequality we may derive the inequality
\begin{align*}
(\log R)^{p(q-1)}Y'(R)R \gtrsim (Y(R))^{pq}.
\end{align*}

Setting $\theta=\varepsilon^{pq}$, $p_1=pq$ and $p_2=1+p(q-1)$, from Lemma \ref{lemma 3.10} we have $T(\varepsilon)\leqslant \exp \big(C \varepsilon^{-q(pq-1)}\big)$ for a suitable positive constant $C$.

\begin{rem} Let us remark explicitly that from the condition $0=F(n+\mu_1,p,q)>F(n+\mu_2,q,p)$ does not follow in general, as for the weakly coupled system of free wave equations, that $p>q$, due to the presence of different shifts in the first argument of $F$.
\end{rem}

\subsubsection{Case $F(n+\mu_1,p,q)<0=F(n+\mu_2,q,p)$} \label{Subsubsection 2}

Proceeding as in the previous section but choosing now  $Y=Y\big[|u|^q \Phi_{\beta_p,\mu_2,\nu_2^2}\,\big]$, where $\beta_p\doteq \frac{n-\mu_2+1}{2}-\frac{1}{p}\,$, it is possible to prove in the case $F(n+\mu_2,q,p)=0$  the upper bound estimate $T\leqslant \exp \big(C \varepsilon^{-p(pq-1)}\big)$ for a suitable positive constant $C$.

\subsubsection{Case $F(n+\mu_1,p,q)=F(n+\mu_2,q,p)=0$}

In this case, combining the results of Sections \ref{Subsubsection 1} and \ref{Subsubsection 2}, it follows immediately the upper bound $T\leqslant \exp \big(C \varepsilon^{-\min\{p(pq-1),q(pq-1)\}}\big)$ for the lifespan. However, we can further improve this estimate. 

First, we prove that $F(n+\mu_1,p,q)=F(n+\mu_2,q,p)=0$ implies
\begin{equation}
\begin{split}
\beta_q -1&= n-\frac{n+\mu_2-1}{2}p \, , \\
\beta_p -1& = n-\frac{n+\mu_1-1}{2}q \, . 
\end{split} \label{beta p and beta q in the double critical case}
\end{equation}
Let us introduce the quantities
\begin{align*}
A\doteq n+1-\beta_p -\tfrac{n+\mu_1-1}{2}q, \quad B\doteq n+1-\beta_q -\tfrac{n+\mu_2-1}{2}p.
\end{align*} Straightforward computations show that
\begin{align*}
Ap+B= (pq-1)F(n+\mu_1,p,q)=0, \quad A+Bq= (pq-1)F(n+\mu_2,q,p)=0.
\end{align*} Hence, since $pq>1$, we have immediately $A=B=0$, that implies in turn the validity of \eqref{beta p and beta q in the double critical case}.

Let us consider $Y=Y\big[|v|^p\Phi_{\beta_q,\mu_1,\nu_1^2}\, \big]$ as in Section \ref{Subsubsection 1}. Due to the assumption $F(n+\mu_1,p,q)=0$, it holds \eqref{Y log estimate} as in Section \ref{Subsubsection 1}. The next step is to improve \eqref{Y lower bound}. Using \eqref{beta p and beta q in the double critical case}, we may rewrite \eqref{(Zwei)} and \eqref{(Eins)} as follows
\begin{align*}
\int_0^T \int_{\mathbb{R}^n} |u(t,x)|^q \psi_R^*(t) \, dx \, dt \gtrsim \varepsilon^q R^{\beta_p-1},  \quad
\int_0^T \int_{\mathbb{R}^n} |v(t,x)|^p \psi_R^*(t) \, dx \, dt \gtrsim \varepsilon^p R^{\beta_q-1}.  
\end{align*}
Consequently, 
\begin{align*}
\int_0^T \int_{\mathbb{R}^n} |u(t,x)|^q \, \Phi_{\beta_p,\mu_2,\nu_2^2}(t,x) \, \psi_R^*(t) \, dx \, dt \gtrsim \varepsilon^q ,  \quad
\int_0^T \int_{\mathbb{R}^n} |v(t,x)|^p \, \Phi_{\beta_q,\mu_1,\nu_1^2}(t,x) \, \psi_R^*(t) \, dx \, dt \gtrsim \varepsilon^p . 
\end{align*} Also, we proved
\begin{align*}
\begin{cases}
Y'(R)R \gtrsim \varepsilon^p , \\
(\log R)^{p(q-1)}Y'(R)R \gtrsim (Y(R))^{pq}.
\end{cases}
\end{align*} Applying Lemma \ref{lemma 3.10} with $\theta=\varepsilon^p$ and $p_1=pq,p_2=1+p(q-1)$, we get the estimate $T(\varepsilon)\leqslant \exp\big(C\varepsilon^{-(pq-1)}\big)$.

\subsubsection{Case $F(n+\mu_1,p,q)=F(n+\mu_2,q,p)=0$ with the same scale-invariant coefficients in the linear part}

In this last case we assume that $\mu_1=\mu_2\doteq \mu$ and $\nu_1^2=\nu_2^2\doteq \nu^2$. As we have the same shift, then,
 the condition $F(n+\mu,p,q)=F(n+\mu,q,p)=0$ implies $p=q=p_0(n+\mu)$. Therefore, $w=u+v$ is a super-solution of \eqref{scale inv equation semi c} with $c=2^{-p}$. Hence, Proposition \ref{Prop in the crit case semi eq} implies $T(\varepsilon)\leqslant \exp\big(C \varepsilon^{-p(p-1)}\big)$. This completes the proof of Theorem \ref{Thm critical case}.

\begin{rem} Let us underline that the sign assumptions on $u_0,u_1,v_0,v_1$ in Theorem \ref{Thm critical case} can be weakened. Indeed, instead of assuming the nonnegativity of these functions, it is sufficient to require that
\begin{align*}
I_{\mu_1,\nu_1^2}[u_0,u_1],  I_{\mu_2,\nu_2^2}[v_0,v_1]>0 \qquad \mbox{and} \qquad J_{\beta_q,\mu_1,\nu_1^2}[u_0,u_1],  J_{\beta_p,\mu_2,\nu_2^2}[v_0,v_1]>0,
\end{align*} as we have seen throughout the proof.
\end{rem}

\section{Final remarks}

According to the blow-up results that are proved in this paper, it is natural to conjecture that for nonnegative and small $\delta_1,\delta_2$ (for example, at least for $n\geqslant 3$ and $0\leqslant \delta_1,\delta_2\leqslant  (n-2)^2$ when \eqref{technical restrictions on p,q critical case} is always fulfilled) the critical curve for \eqref{weakly coupled system} is given by \eqref{critical exponent wave like case system, critical case}. Even though the existence of global in time small data solutions in the supercritical case is an open problem, some partial results for the single semilinear equation \eqref{scale inv eq} in the case $\delta=1$ (cf. \cite{Pal18odd,Pal18even}) suggest the likelihood and plausibility of this conjecture.

 In the case in which instead of scale-invariant damping terms (in the massless case though) we consider time-dependent coefficients for the damping terms in the scattering case (see \cite{Wirth04,Wirth06,Wirth07} for the classification of a damping term for a wave model with time-dependent coefficient), the presence of these damping terms has no influence on the critical curve.  Indeed, in a series of forthcoming papers \cite{PalTak19,PalTak19dt,PalTak19mix} several blow-up results for weakly coupled systems of damped wave equations in the scattering case with different type of nonlinearities are proved. In particular, in the case of power nonlinearities the corresponding critical curve will be exactly the same one as for the weakly coupled system of semilinear not-damped wave equations with the same nonlinearities, that is, \eqref{crit curv weakly coupled system wave}. This fact proves, once again, how the time-dependent and scale-invariant coefficients for lower order terms in a wave model make it a threshold model between ``parabolic-like''  and ``hyperbolic-like'' models. A further peculiar characteristic of scale-invariant models is that the multiplicative constants in the time-dependent coefficients (that is, $\mu_1,\mu_2,\nu_1^2,\nu_2^2$ for the weakly coupled system in \eqref{weakly coupled system}) determine the analytic expression of the critical condition for the exponents of the nonlinear terms with the presence of shifts in comparison to the corresponding critical condition for the related semilinear wave or damped wave model.

\section*{Acknowledgments}

The author is member of the Gruppo Nazionale per L'Analisi Matematica, la Probabilit\`{a} e le loro Applicazioni (GNAMPA) of the Instituto Nazionale di Alta Matematica (INdAM). The author thanks Michael Reissig (TU Freiberg) and Hiroyuki Takamura (Tohoku University) for helpful discussions on the topic.

\vspace*{0.5cm}





\end{document}